\documentclass[a4paper,11pt]{amsart}

\usepackage[utf8]{inputenc}
\usepackage[T1]{fontenc}

\usepackage{amsmath}
\usepackage{amssymb}
\usepackage{thmtools}
\usepackage{hyperref}
\usepackage[capitalize]{cleveref}

\Crefname{figure}{Figure}{Figures}

\crefformat{equation}{\textup{#2(#1)#3}}
\crefrangeformat{equation}{\textup{#3(#1)#4--#5(#2)#6}}
\crefmultiformat{equation}{\textup{#2(#1)#3}}{ and \textup{#2(#1)#3}}
{, \textup{#2(#1)#3}}{, and \textup{#2(#1)#3}}
\crefrangemultiformat{equation}{\textup{#3(#1)#4--#5(#2)#6}}%
{ and \textup{#3(#1)#4--#5(#2)#6}}{, \textup{#3(#1)#4--#5(#2)#6}}{, and \textup{#3(#1)#4--#5(#2)#6}}

\Crefformat{equation}{Equation~\textup{#2(#1)#3}}
\Crefrangeformat{equation}{Equations~\textup{#3(#1)#4--#5(#2)#6}}
\Crefmultiformat{equation}{Equations~\textup{#2(#1)#3}}{ and \textup{#2(#1)#3}}
{, \textup{#2(#1)#3}}{, and \textup{#2(#1)#3}}
\Crefrangemultiformat{equation}{Equations~\textup{#3(#1)#4--#5(#2)#6}}%
{ and \textup{#3(#1)#4--#5(#2)#6}}{, \textup{#3(#1)#4--#5(#2)#6}}{, and \textup{#3(#1)#4--#5(#2)#6}}

\crefname{section}{section}{sections}

\usepackage{tikz}
\tikzset{normalnode/.style={circle, draw, fill=black, inner sep=0, minimum width=1.5mm}}

\newcommand{\iso}{\cong}

\newcommand{\cay}{\text{Cay}}

\newcommand{\card}[1]{|{#1}|}

\DeclareMathOperator*{\biglor}{\bigvee}

\newcommand{\Nn}{\mathbb{N}}
\newcommand{\Zz}{\mathbb{Z}}
\newcommand{\Pc}{\mathcal{P}}
\newcommand{\Xc}{\mathcal{X}}
\newcommand{\Yc}{\mathcal{Y}}
\newcommand{\Zc}{\mathcal{Z}}
\newcommand{\Rc}{\mathcal{R}}
\newcommand{\Cc}{\mathcal{C}}

\newtheorem{theorem}{Theorem}[section]
\newtheorem{proposition}[theorem]{Proposition}

\newtheorem{corollary}[theorem]{Corollary}
\newtheorem{conjecture}[theorem]{Conjecture}

\newtheorem{lemma}[theorem]{Lemma}

\theoremstyle{definition}

\theoremstyle{remark}
\newtheorem{remark}[theorem]{Remark}
\newtheorem{claim}[theorem]{Claim}

\DeclareMathOperator{\Aut}{Aut}
\DeclareMathOperator{\Perm}{\mathfrak{S}}
\DeclareMathOperator{\tww}{tww}
\DeclareMathOperator{\stww}{stww}
\DeclareMathOperator{\utww}{utww}
\DeclareMathOperator{\gv}{gn}  
\DeclareMathOperator{\qn}{qn}
\DeclareMathOperator{\sqn}{sqn}  
\DeclareMathOperator{\uqn}{uqn} 
\newcommand{\diam}{\textsf{diam}}
\newcommand{\girth}{\textsf{girth}}
\newcommand{\eqdef}{\stackrel{\text{def}}{=}}
\newcommand{\setst}[2]{\textstyle \left\{#1\ |\ #2\right\}}
\newcommand{\bij}[1]{M_{#1}} 
\newcommand{\adj}[2]{A_{#1}(#2)} 
\newcommand{\grppres}[2]{\left\langle #1 ; #2 \right\rangle}
\newcommand{\symdiff}{\triangle}

\renewcommand{\le}{\leqslant}
\renewcommand{\ge}{\geqslant}

\title{Twin-width VII: Groups}

\author[É.~Bonnet]{Édouard Bonnet}
\author[C.~Geniet]{Colin Geniet}
\author[R.~Tessera]{Romain Tessera}
\author[S.~Thomassé]{Stéphan Thomassé}

\newcommand{\lipaddr}{Laboratoire de l'Informatique du Parallélisme, ENS de Lyon, 46 allée d’Italie, 69364 Lyon CEDEX 07, France}

\address{\lipaddr}
\email{edouard.bonnet@ens-lyon.fr}

\address{\lipaddr}
\email{colin.geniet@ens-lyon.fr}

\address{Institut de Mathématiques de Jussieu-Paris Rive Gauche, Place Aurélie Nemours, 75013 Paris, France}
\email{romain.tessera@imj-prg.fr}

\address{\lipaddr}
\email{stephan.thomasse@ens-lyon.fr}

\thanks{%
  The first, second, and fourth authors were partially supported by the ANR grant TWIN-WIDTH (ANR-21-CE48-0014-01)}



\begin{document}
\begin{abstract}
  Twin-width is a recently introduced graph parameter with applications in algorithmics, combinatorics, and finite model theory.
  For graphs of bounded degree, finiteness of twin-width is preserved by quasi-isometry.
  Thus, through Cayley graphs, it defines a group invariant.

  We prove that groups which are abelian, hyperbolic, ordered, solvable, or with polynomial growth, have finite twin-width.
  Twin-width can be characterised by excluding patterns in the self-action by product of the group elements.
  Based on this characterisation, we propose a strengthening called \emph{uniform twin-width},
  which is stable under constructions such as group extensions, direct products, and direct limits.

  The existence of finitely generated groups with infinite twin-width is not immediate.
  We construct one using a result of Osajda on embeddings of graphs into groups.
  This implies the existence of a class of finite graphs with unbounded twin-width
  but containing~$2^{O(n)} \cdot n!$ graphs on vertex set~$\{1,\dots,n\}$,
  settling a question asked in a previous work.
\end{abstract}
\maketitle

\section{Introduction}
Twin-width is an invariant of graphs and matrices introduced in~\cite{twin-width1},
with applications in algorithmics, logic, enumerative combinatorics~\cite{twin-width2,twin-width3,twin-width4}.
Finite twin-width is known to be a group invariant of Cayley graphs~\cite[Section~7]{twin-width2}.
Our main result is the construction of groups with infinite twin-width,
using an embedding theorem of Osajda~\cite{osajda2020small}.
\begin{restatable}{theorem}{infinitetww}
  \label{thm:infinite-tww}
  There is a finitely generated group with infinite twin-width.
\end{restatable}
This disproves the `small conjecture' proposed in~\cite{twin-width2},
related to the Stanley-Wilf conjecture proved by Marcus and Tardos~\cite{MarcusT04}, see~\cref{sec:intro-growth} below.
It is also related to the problem of finding graphs with bounded degree and infinite twin-width, see~\cref{sec:intro-tww-bounded-degree}.

We study characterisations, variants, and stability properties of twin-width of groups,
and prove that a number of well-known classes of groups have finite twin-width, see~\cref{sec:intro-tww-groups}.
Twin-width of groups has an interesting relation to queue number, a more restrictive graph parameter:
a number of results of this paper also hold for the queue number, see~\cref{sec:intro-queue-number}.

\subsection{Growth of classes excluding substructures}
\label{sec:intro-growth}
Marcus and Tardos proved that any~$n \times n$ 0--1 matrix with more than~$c_k n$ 1-entries contains all~$k \times k$ permutation matrices as submatrices,
for an appropriate~$c_k$ function of~$k$ only~\cite{MarcusT04}.
A major corollary is the Stanley-Wilf conjecture: For any fixed permutation~$\sigma$,
the class of permutations avoiding~$\sigma$ as pattern has single-exponential growth---%
as opposed to the class of all permutations which has factorial growth.
Guillemot and Marx built on the result of Marcus and Tardos by giving a characterisation of classes of permutations which exclude some pattern~\cite{Guillemot14}:
they have bounded \emph{width}, which is defined by representing permutations with points in the plane,
iteratively \emph{contracting} these points down to a single one
while measuring a number of \emph{conflicts} throughout this contraction sequence,
and minimizing the number of conflicts over all choices of contraction sequence.

A natural generalisation of the Stanley-Wilf conjecture to arbitrary 0--1 matrices is the following:
Any class of matrices~$\Cc$ closed under taking submatrices has growth at most single-exponential or at least factorial.
This was established in~\cite{twin-width4}, by proving that classes of matrices with exponential growth
are exactly the classes with bounded \emph{twin-width}, a natural generalisation of Guillemot and Marx's width to graphs and matrices proposed in~\cite{twin-width1}.
Matrices with bounded twin-width are also characterised by excluding \emph{grids of high rank}, which can be seen as a generalisation of excluding patterns.

The logical next question in this line of work is the generalisation to graphs.
One may think of graphs as essentially the same as a 0--1 matrix, through adjacency matrices.
But a major difference is that matrices have an innate order, which is lost in graphs.
For graphs, the sensible equivalent of classes with exponential growth is \emph{small classes}:
classes of graphs which contain at most~$n! \cdot c^n$ graphs on vertices~$1,\dots,n$, for some constant $c$.
The additional $n!$ compensates for the choice of a permutation of the vertices.
Well known examples of small classes of graphs include trees, planar graphs,
and more generally proper minor-closed classes~\cite{Norine06}.
These examples are generalised by the following.
\begin{theorem}[\cite{twin-width2}]
  \label{thm:tww-small}
  Classes of graphs with bounded twin-width are small.
\end{theorem}

With this result, and the cases of matrices and permutations in mind,
the following very natural conjecture was raised.
\begin{conjecture}[Small conjecture~\cite{twin-width2}]
  \label{conj:small-conjecture}
  A class of graphs closed under induced subgraphs is small if and only if it has bounded twin-width.
\end{conjecture}
We disprove \cref{conj:small-conjecture} using an approach suggested in~\cite{twin-width2}:
if~$G$ is a Cayley graph of some group~$\Gamma$,
then the class of finite subgraphs of~$G$ is small~\cite[Lemma~8.1]{twin-width2}.
If~$\Gamma$ is a group with infinite twin-width, then this class furthermore has unbounded twin-width,
hence is a counterexample to \cref{conj:small-conjecture}.
Thus \cref{thm:infinite-tww} gives:
\begin{corollary}
  \label{cor:small-conjecture-false}
  There is a small class of graphs closed under induced subgraphs and with unbounded twin-width.
\end{corollary}

\subsection{Properties of twin-width of groups}
\label{sec:intro-tww-groups}
We prove that for bounded degree graphs, finiteness of twin-width is preserved by \emph{quasi-isometries},
i.e.\ maps which preserve distances up to some affine lower and upper bounds,
and more generally by \emph{coarse} and \emph{regular} embeddings.
This is a generalisation of~\cite[Lemma~8.2]{twin-width2},
which proved that finite twin-width is a group invariant of Cayley graphs:
given a group~$\Gamma$, all Cayley graphs of~$\Gamma$ are quasi-isometric,
hence either all or none of them have finite twin-width.
In relation with coarse geometry, we remark that Gromov-hyperbolic groups have finite twin-width,
because they embed into finite products of trees with bounded degree~\cite{buyalo2007embedding}.

Groups with finite twin-width have a remarkable second characterisation:
a group~$\Gamma$ has finite twin-width if and only if there is a total order~$<$ on~$\Gamma$
such that for any~$x \in \Gamma$, the action of~$x$ on~$(\Gamma,<)$ by product
is a permutation with finite \emph{width} (in the sense of Guillemot and Marx~\cite{Guillemot14}),
or equivalently a permutation which avoids some pattern.
For example, if~$<$ is a right-invariant order on~$\Gamma$,
the action of any~$x$ is an increasing permutation,
i.e.\ a permutation which excludes a decreasing sequence of length 2 as pattern.
This characterisation extends to non-finitely generated groups,
and also suggests a natural strengthening called \emph{uniform twin-width}:
in the above statement, we require a uniform bound on the width of actions of elements of $\Gamma$,
or equivalently require all these actions to avoid the same pattern.
Having finite uniform twin-width implies finite twin-width, while the converse does not a priori hold.

Uniform twin-width is preserved by a number of group operations.
\begin{lemma}[\cref{lem:extension-gen,lem:direct-product,lem:group-limit,lem:wreath-product,lem:finite-quotient} in the text]
  \label{lem:operations-summary}
  The following constructions preserve finiteness of uniform twin-width:
  \begin{enumerate}
    \item \label{item:extension} Group extensions (and thus Cartesian or semidirect products);
    \item \label{item:direct-product} Infinite direct products and direct sums;
    \item \label{item:limit} Direct limits;
    \item \label{item:subgroup} Taking supergroups with finite index;
    \item \label{item:quotient} Quotienting by finite subgroups;
    \item \label{item:wreath} Wreath products.
  \end{enumerate}
\end{lemma}
From this and a few other results, we obtain many examples of groups with finite uniform twin-width:
abelian groups, solvable groups, groups with polynomial growth, lamplighter groups, automorphism groups of trees.

We conjecture that uniform twin-width is strictly stronger than twin-width,
i.e.\ there are groups with finite twin-width and infinite uniform twin-width.
A candidate is the group of finitely supported permutations on~$\Zz$:
it has finite twin-width, and we think it is unlikely to have finite uniform twin-width
as this would imply a universal bound on the uniform twin-width of finite groups,
and more generally elementary amenable groups.

\subsection{Twin-width of bounded degree graphs}
\label{sec:intro-tww-bounded-degree}
Bounded degree graphs are an interesting and not very well understood special case in which to consider twin-width.
It is known that graphs with bounded degree can have infinite twin-width by the counting argument of the previous section:
it is easy to verify that the class of 3-regular graphs is not small,
hence it contains finite graphs with arbitrarily large twin-width by \cref{thm:tww-small}---%
in fact, for any $k$, almost all 3-regular graphs have twin-width more than~$k$.
On the other hand, no explicit or deterministic construction
for bounded degree graphs with unbounded twin-width is known.
Finding such graphs could significantly help our understanding of twin-width.
Bilu-Linial expanders~\cite{Bilu06} are an interesting example:
they are 3-regular expanders, hence seemingly complex graphs,
but with a structured construction which allows to prove that
they have twin-width at most~6~\cite[Section~5]{twin-width2}.

If~$\Gamma$ is a group with infinite twin-width, the class of subgraphs of some Cayley graph of~$\Gamma$
has bounded degree and unbounded twin-width, which could solve the aforementioned problem.
Unfortunately, the group constructed in \cref{thm:infinite-tww} does not help in this regard:
its construction starts from a sequence of bounded degree graphs
with unbounded twin-width in the first place, and relies largely on probabilistic methods.
Thus the same question remains for groups:
can we find a more explicit construction of a group with infinite twin-width,
or one with a direct proof that twin-width is infinite?

\subsection{Other graph parameters for groups}
\label{sec:intro-queue-number}
This paper proposes to consider twin-width on groups.
It is natural to ask whether the same can be done with other classical graph parameters.

Tree-width---arguably the most important graph `width' parameter---%
is a group invariant, studied by Kuske and Lohrey~\cite{kuske2005logical}.
Using that graphs with bounded tree-width and bounded degree can be partitioned
into parts of bounded size so that the quotient is a tree~\cite{ding1995decomposition},
they prove the equivalence between finite tree-width, virtually free,
context-free word problem, and the Cayley graphs having decidable monadic second-order theory.

Stack and queue number are two other graph parameters, sitting between tree-width and twin-width:
bounded tree-width implies bounded stack and queue number~\cite{ganley2001pagenumber,dujmovic2005layout},
either of which implies bounded twin-width~\cite[Theorem~7.4]{twin-width2}.
Eppstein et al.\ recently proved that the 3-dimensional grid with diagonals
has infinite stack number~\cite{eppstein2022threedimensional}.
This is the first known explicit construction of bounded degree graphs with unbounded stack number
(probabilistic constructions were previously known~\cite{malitz1994pagenumber}).
This implies that stack number is not a group invariant,
because the 3-dimensional grid \emph{without} diagonals has finite stack number,
and both variants of the grid are Cayley graphs of~$\Zz^3$.

On the other hand, queue number can be shown to be a group invariant,
and a surprisingly large part of the results of this paper can be adapted to queue number.
The group of \cref{thm:infinite-tww} of course has infinite queue number,
since finite queue number implies finite twin-width.
Queue number is characterised by excluding decreasing sequences in the adjacency matrix
for some choice of vertex ordering.
This somewhat resembles the matrix characterisation of twin-width,
and allows to define a \emph{uniform} queue number for groups.
Most operations of \cref{lem:operations-summary} also preserve finite uniform queue number.
However, unlike for twin-width, products and extensions preserve finiteness of uniform queue number,
but not the exact bound on the latter.
Thus, we still obtain for instance that finitely generated abelian groups have finite uniform queue number,
but the proof does not give a uniform bound on the latter,
and thus does not go through direct limits to reach non-finitely generated abelian groups.
Despite this, we do not have examples of groups with finite twin-width and infinite queue number
(or similarly for the uniform variants).
Separating twin-width and queue number on sparse graphs is an open problem.

\subsection{Organisation of the paper}
\Cref{sec:notation} summarises basic notions and notations used in this paper.
\Cref{sec:twinwidth} presents the definition of twin-width, and known results used in this work.
Because this paper focuses solely on graphs of bounded degree,
we give a more restricted definition than the one of~\cite{twin-width1},
which slightly simplifies several of these results.

\Cref{sec:geometric-tww} presents twin-width as a quasi-isometric invariant property of Cayley graph,
giving some basic examples, as well as examples related to geometric group theory.
\Cref{sec:utww} introduces a second characterisation of twin-width of groups using matrices,
and a natural strengthening: uniform twin-width.
\Cref{sec:finite-tww} proves a number of stability results for (uniform) twin-width under common group operations,
and uses these results to show that several well-known classes of groups have finite twin-width.

\Cref{sec:infinite-tww} constructs a finitely generated group with infinite twin-width,
by embedding a well-chosen sequence of graphs of bounded degree into a~group.
This construction gives a recursive presentation, with decidable word problem.
It follows that there are finitely presented groups with infinite twin-width.
This can be read independently of \cref{sec:utww,sec:finite-tww}.

Finally, \cref{sec:queue-number} discusses the relation of twin-width and queue number in groups.
We show that many results of this paper also hold when replacing twin-width with queue number,
and we point out the points where this replacement fails.

\section{Preliminaries and notations}
\label{sec:notation}
\subsection{Graphs}
We consider undirected graphs, without loops nor multiple edges, and which may be finite or countably infinite.
A graph is denoted as a pair $G = (V,E)$ with~$V$ and~$E$ the sets of vertices and edges respectively,
and we also write~$V(G)$ for~$V$, and~$E(G)$ for~$E$.

If~$G,H$ are graphs, $H$ is a \emph{subgraph} of~$G$ if $V(H) \subset V(G)$ and $E(H) \subset E(G)$.
It is an \emph{induced subgraph} if $V(H) \subset V(G)$,
and~$E(H)$ consists exactly of the edges in~$E(G)$ with endpoints in~$V(H)$.
A \emph{class} of graphs is a collection of graphs which is closed under isomorphism.
A class is \emph{hereditary} if it is also closed under taking induced subgraphs.

If $v \in V(G)$, its degree~$d(v)$ is the number of edges incident to~$v$.
We denote by~$\Delta(G)$ the maximum degree of vertices in~$G$.
A graph is $k$-\emph{regular} if all its vertices have degree~$k$.
A \emph{cubic} graph is a 3-regular graph.

The \emph{diameter} of a connected graph~$G$, denoted~$\diam(G)$, is the maximum distance between two vertices of~$G$.
The \emph{girth} of a graph~$G$, denoted~$\girth(G)$, is the minimum length of a cycle in~$G$.

\subsection{Matrices}
We consider binary, possibly infinite matrices,
and pay particular attention to the order of their rows and columns.
That is, if $(X,<_X)$, $(Y,<_Y)$ are two totally ordered, possibly infinite sets,
an $X \times Y$ binary matrix is a map $M : X \times Y \to \{0,1\}$.
Equivalently, it is a binary relation between ordered sets.
We call~$X$, respectively~$Y$, the set of column, respectively row, indices.
The matrix~$M$ is finite if~$X,Y$ are finite.
A \emph{submatrix} of~$M$ is obtained by restricting~$M$ to $X' \times Y'$ for some $X' \subset X, Y' \subset Y$.
By analogy with graphs, we call degree of a column or a row of a matrix the number of~1-entries it contains.

Given $f : X \to Y$, the matrix~$\bij f$ of~$f$ is defined by $\bij f (x,y) = 1$ if and only if $f(x) = y$.
We will primarily consider this construction when~$f$ is bijective,
in which case we call~$\bij f$ a \emph{bijection matrix}.
Bijection matrices are exactly the matrices with all rows and columns of degree 1.
In the finite case, they are more commonly called permutation matrices,
but this would be abusive in general since~$X$ and~$Y$ need not be isomorphic as orders.

Given a graph $G = (V,E)$ and a total order~$<$ on~$V$,
the adjacency matrix~$\adj{<}{G}$ is the $(V,<) \times (V,<)$ matrix
such that $\adj{<}{G}(x,y) = 1$ if and only if~$x,y$ are adjacent in~$G$.

\subsubsection{Divisions of matrices}
If~$M$ is an $X \times Y$ matrix, a $k \times l$\emph{-division} of~$M$ consists of
partitions of rows and columns into intervals;
i.e.\ it is given by $\Xc = \{X_1, \dots, X_k\}$ partition of~$X$, $\Yc = \{Y_1, \dots, Y_l\}$ partition of~$Y$,
such that every $X_i$ or $Y_j$ is an interval, and with $X_1 <_X \dots <_X X_k$ and $Y_1 <_Y \dots <_Y Y_l$.
The \emph{zone}~$(i,j)$ of the division~$(\Xc,\Yc)$ is the submatrix~$M$ restricted to~$X_i \times Y_j$.
The \emph{quotient} matrix~$M/(\Xc,\Yc)$ is the~$\Xc \times \Yc$ matrix
with a~`1' at position~$X_i,Y_j$ if and only if the corresponding zone~$X_i \times Y_j$ of~$M$
is non-zero (i.e.\ contains at least a~`1').

\subsection{Groups}
Unless stated otherwise, we consider Cayley graphs, group actions, etc.\ on the right.
If~$\Gamma$ is a group and~$S \subset \Gamma$ a finite generating set, its Cayley graph is
\[ \cay(\Gamma, S) \eqdef \left( \Gamma, \setst{xy}{x^{-1}y \in S \cup S^{-1}} \right). \]
Thus edges go from~$x$ to~$y$ by multiplying on the right by a generator.

A group presentation consists of a set~$S$ of generators, and a set~$R$ of relators
which are finite words over~$S \uplus S^{-1}$, where~$S^{-1}$ denotes formal inverses of elements of~$S$.
The group associated to this presentation, denoted~$\grppres{S}{R}$,
is the free group over~$S$ quotiented by the normal subgroup generated by~$R$.
A group~$G$ is said to admit~$\grppres{S}{R}$ as presentation if it is isomorphic to~$\grppres{S}{R}$.
A group is finitely presented if it admits a presentation with finite sets of generators and of relators.

For a group~$\Gamma$ finitely generated by~$S$, the word problem is the following decision problem:
given~$w$ a word over~$S \cup S^{-1}$, is~$w$ interpreted as an element of~$\Gamma$ equal to the neutral element?
This problem can be undecidable even for finitely presented groups.
Decidability of the word problem is well known to be independent of the choice of~$S$.

\subsection{Coarse geometry}
In coarse geometry, distances in metric spaces are considered up to some possibly large error.
The fundamental notion is the quasi-isometry.
For~$\lambda \ge 0$, a map $f : X \to Y$ between metric spaces is
a~\emph{$\lambda$-quasi-isometric embedding} if for all $x,x' \in X$,
\[ \lambda^{-1} d_X(x,x') - \lambda \le d_Y\left( f(x),f(x') \right) \le \lambda d_X(x,x') + \lambda \]
Thus quasi-isometric embeddings are maps which preserve distances up to affine lower and upper bound.
The map~$f$ is a \mbox{\emph{$\lambda$-quasi-isometry}} if it is also \mbox{$\lambda$-cobounded},
i.e.\ every~$y \in Y$ is at distance at most~$\lambda$ of~$f(X)$.
The spaces~$X$ and~$Y$ are said to be quasi-isometric
if there is a $\lambda$-quasi-isometry $X \to Y$ for some~$\lambda$.
This is equivalent to having back and forth quasi-isometric embeddings between $X$ and $Y$
whose compositions are within a bounded distance of identity maps.
Quasi-isometry is an equivalence relation.

In a graph~$G$, the vertex set~$V(G)$ is equipped with the shortest path distance~$d_G(x,y)$.
This defines a metric space when~$G$ is connected.
Remark that this metric space is quasi-isometric to the geometric realisation of~$G$,
in which one replaces each edge of~$G$ by a path isometric to the real segment~$[0,1]$.

Given a group~$\Gamma$ and~$S,S'$ two finite generating sets of~$\Gamma$,
it is easy to verify that $\cay(\Gamma,S)$ and $\cay(\Gamma,S')$ are quasi-isometric.
This is the foundation of coarse geometry on groups.

\emph{Regular embeddings} are another notion of embedding in coarse geometry,
which is particularly interesting for graphs seen as metric spaces.
A map $f : X \to Y$ is a $\lambda$-regular embedding if it is $\lambda$-Lipschitz
and for all $y \in Y$, $\card{f^{-1}(y)} \le \lambda$.
If all balls of~$X$ have cardinality bounded by a function of their radius,
then any quasi-isometric embedding~$X \to Y$ is also a regular embedding.
In particular, this is the case when~$X$ is a graph with bounded degree.

\subsection{Direct limits}
Let~$S$ be some structure (in our case, a group or a graph),
and let~$(S_i)_{i \in I}$ be a family of substructures of~$S$.
The family~$(S_i)_{i \in I}$ is \emph{directed} if for all~$i,j \in I$,
there exists~$k \in I$ with~$S_i \cup S_j \subseteq S_k$.
The structure~$S$ is said to be direct limit of~$(S_i)_{i \in I}$
if~$(S_i)_{i \in I}$ is directed and~$S = \bigcup_{i \in I} S_i$. 
For instance, chains are a very special case of directed families,
a graph is always direct limit of all its finite induced subgraphs,
and similarly a group is direct limit of its finitely generated subgroups.

The following is a folklore lemma which we use to show that twin-width goes through direct limits.
We give a proof with ultrafilter arguments.
When the directed family is isomorphic to~$\Nn$ as order,
one can also prove it using Kőnig's theorem, or an induction and the axiom of countable choice.
\begin{lemma}
  \label{lem:direct-limit-order}
  Let~$S$ be direct limit of~$(S_i)_{i \in I}$,
  and for each~$i \in I$ let~$<_i$ be a total order on~$S_i$.
  Then there exists a total order~$<$ on~$S$ such that
  for any finite~$X \subset S$, there is some~$i \in I$ with~$X \subseteq S_i$
  and such that~$<$ and~$<_i$ coincide on~$X$.
\end{lemma}
\begin{proof}
  \newcommand{\upclos}{\uparrow\!}
  Order~$I$ by~$i \le j$ if~$S_i \subseteq S_j$.
  For~$i \in I$, define its upward closure $\upclos i = \setst{j}{j \ge i}$.
  Consider the family of non-empty and eventually upward-closed subsets
  \[ \mathcal{F} \eqdef \setst{J \subseteq I}{\exists i \in I,~\upclos i \subseteq J}. \]
  If~$A \in \mathcal{F}$ and~$A \subset B$, then clearly~$B \in \mathcal{F}$.
  Furthermore, given~$A_1,A_2 \in \mathcal{F}$, consider~$i_1,i_2$ such that~$\upclos i_j \subseteq A_j$.
  Since~$(S_i)_{i \in I}$ is directed, there exists~$k \in I$ such that~$i_1,i_2 \le k$.
  Then~$\upclos k \subseteq A_1 \cap A_2$, and thus~$A_1 \cap A_2 \in \mathcal{F}$.
  This proves that~$\mathcal{F}$ is a proper filter,
  hence there is an ultrafilter~$\mathcal{U}$ containing~$\mathcal{F}$.

  For~$x \neq y \in S$, consider the indices which order~$x$ before~$y$.
  \[ I_{xy} \eqdef \setst{i \in I}{x <_i y} \]
  The limit order on $S$ in then defined as $x < y$ if and only if $I_{xy} \in \mathcal{U}$.
  Since~$S$ is direct limit of~$(S_i)_{i \in I}$, for any~$x,y \in S$, there is some~$i \in I$ such that~$x,y \in S_i$.
  Assuming~$x \neq y$, it is clear that $\upclos i \subseteq I_{xy} \cup I_{yx}$:
  for any~$j \ge i$, $x$ and~$y$ must be comparable by $<_j$.
  Since~$\mathcal{U}$ is an ultrafilter, it follows that either~$I_{xy} \in \mathcal{U}$ or~$I_{yx} \in \mathcal{U}$
  (but not both as they are disjoint).
  Thus~$<$ is a total, irreflexive and antisymmetric relation.
  Let us prove the condition on~$<$ required by the lemma---transitivity follows from it.

  Given~$X \subseteq S$ any finite subset, there is some~$i \in I$ with~$X \subseteq S_i$.
  Define
  \[ A \eqdef (\upclos i) \cap \bigcap_{\substack{x,y \in X \\ x < y}} I_{xy} \]
  Then~$A$ is a finite intersection of elements of~$\mathcal{U}$,
  hence~$A \in \mathcal{U}$ and in particular it is non-empty.
  Choose~$j \in A$.
  Then~$j \ge i$, hence~$X \subseteq S_j$.
  Furthermore for any~$x < y$ in~$X$, we have~$j \in I_{xy}$ hence~$x <_j y$.
  Since both~$<$ and~$<_j$ are total, irreflexive, and antisymmetric on~$X$,
  this implies that they coincide on~$X$ as desired.
\end{proof}

\section{Twin-Width of graphs and matrices}
\label{sec:twinwidth}
This section recalls the definition of twin-width~\cite{twin-width1}, and some fundamental results.
We will present a simplified definition specialised for graphs of bounded degree, called \emph{strict twin-width}.
The two notions are equivalent on graphs of bounded degree,
in the sense that they differ by at most the maximum degree of the graph.
When considering matrices, strict twin-width is closely related to the work of Guillemot and Marx
who first introduced twin-width for permutation matrices~\cite{Guillemot14}.
The reader may refer to \cite{twin-width1} for the general definition.
See also \cite[Section~7]{twin-width2} for a discussion on twin-width for sparse graphs beyond bounded degree.

\subsection{Definitions and characterisations}
\subsubsection{Twin-width of graphs}
Let $G=(V,E)$ be a finite graph and~$\Pc$ be a partition of~$V(G)$.
The \emph{quotient graph}~$G/\Pc$ has vertex set~$\Pc$,
and its edges are all~$XY$ with $X,Y\in \Pc$ such that there exists an edge~$xy$ of~$G$ with~$x\in X$ and~$y\in Y$.
A \emph{partition sequence} is a sequence $\Pc_n,\dots,\Pc_{1}$ of partitions of~$V(G)$
such that $\Pc_n=\setst{\{v\}}{v \in V}$ is the partition into singletons,
$\Pc_{1}=\{V\}$ is the trivial partition,
and each~$\Pc_{i-1}$ is obtained by merging two parts of~$\Pc_i$.
The \emph{width} of this partition sequence is
the maximum degree of the quotient graphs, i.e.~$\max_{i=1}^n \Delta(G / \Pc_i)$.
The \emph{strict twin-width} of~$G$, denoted~$\stww(G)$,
is the minimum width of a partition sequence starting with~$G$.

For a partition sequence $\Pc_n,\dots,\Pc_1$,
one may also see the quotients graphs $G/\Pc_n, \dots, G/\Pc_1$ as obtained
by successive contractions of arbitrary pairs of vertices,
starting from~$G$ and ending with a single vertex left, see \cref{fig:contraction-sequence}.
Here, we insist that contracting pairs of non-adjacent vertices is allowed.

\begin{figure}
  \begin{center}
    \begin{tikzpicture}
      \tikzstyle{every node}=[normalnode]
      \def\s{3.3cm}
      \def\t{2cm}

      \foreach \i in {0,...,3}{
        \node (\i0) at (90*\i+45:0.7) {};
        \node (\i1) at (90*\i+45:1.5) {};
      }
      \draw (00) -- (10) -- (20) -- (31) -- (21);
      \draw (31) -- (01) -- (11) -- (21) -- (30) -- (00);
      \draw (10) -- (11);
      \draw (00) -- (20);
      \draw (30) -- (01);
      \draw[thick] (01) circle (0.2);
      \draw[thick] (21) circle (0.2);

      \begin{scope}[xshift=\s]
        \foreach \i in {1,2,3}{
        \node (\i0) at (90*\i+45:0.7) {};
        \node (\i1) at (90*\i+45:1.5) {};
      }
      \node (00) at (45:0.7) {};

      \draw (00) -- (10) -- (20) -- (31) -- (21);
      \draw (11) -- (21) -- (30) -- (00);
      \draw (10) -- (11);
      \draw (00) -- (20);
      \draw[thick] (00) circle (0.2);
      \draw[thick] (20) circle (0.2);
      \end{scope}

      \begin{scope}[xshift=2*\s]
        \foreach \i in {1,2,3}{
        \node (\i0) at (90*\i+45:0.7) {};
        \node (\i1) at (90*\i+45:1.5) {};
      }

      \draw (10) -- (20) -- (31) -- (21);
      \draw (11) -- (21) -- (30) -- (20);
      \draw (10) -- (11);
      \draw[thick] (30) circle (0.2);
      \draw[thick] (31) circle (0.2);
      \end{scope}

      \begin{scope}[yshift=-\s]
        \foreach \i in {1,2}{
        \node (\i0) at (90*\i+45:0.7) {};
        \node (\i1) at (90*\i+45:1.5) {};
      }
      \node (30) at (-45:0.7) {};

      \draw (10) -- (20) -- (30) -- (21) -- (11) -- (10);
      \draw[thick] (20) circle (0.2);
      \draw[thick] (30) circle (0.2);
      \end{scope}

      \begin{scope}[yshift=-\s,xshift=1.5*\t]
        \foreach \i in {1,2}{
        \node (\i0) at (90*\i+45:0.7) {};
        \node (\i1) at (90*\i+45:1.5) {};
      }

      \draw (10) -- (20) -- (21) -- (11) -- (10);
      \draw[thick] (11) circle (0.2);
      \draw[thick] (21) circle (0.2);
      \end{scope}

      \begin{scope}[yshift=-\s,xshift=2.5*\t]
      \node (10) at (135:0.7) {};
      \node (20) at (225:0.7) {};
      \node (3) at (-1.2,0) {};

      \draw (10) -- (20) -- (3) -- (10);
      \draw[thick] (10) circle (0.2);
      \draw[thick] (20) circle (0.2);
      \end{scope}

      \begin{scope}[yshift=-\s,xshift=3.4*\t]
      \node (1) at (0,0) {};
      \node (2) at (-1,0) {};

      \draw (1) -- (2);
      \draw[thick] (1) circle (0.2);
      \draw[thick] (2) circle (0.2);
      \end{scope}

      \begin{scope}[yshift=-\s,xshift=4*\t]
      \node (1) at (0,0) {};
      \end{scope}
    \end{tikzpicture}
  \end{center}
  \caption{%
    A contraction sequence of width 3 (i.e.\ the quotient graphs of a partition sequence),
    represented in reading order. At each step, vertices to be contracted are circled.
  }
  \label{fig:contraction-sequence}
\end{figure}
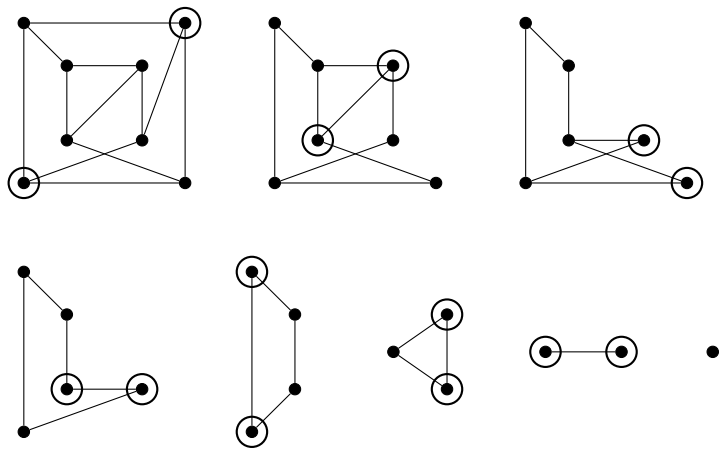

Remark that~$\stww(G)$ is at least~$\Delta(G)$.
Strict twin-width is also non-increasing when taking subgraphs.%
\footnote{%
  Non-strict twin-width is only stable under taking \emph{induced} subgraphs, and has no relation to~$\Delta(G)$.
}
Strict twin-width has the following relation to the general notion of twin-width,
denoted~$\tww(G)$, as defined in~\cite{twin-width1}.
\begin{equation}
  \max\left(\Delta(G), \tww(G)\right) \le \stww(G) \le \Delta(G) + \tww(G)
\end{equation}
That is, the two notions are equivalent up to an additive constant on graphs of bounded degree.
We will only use strict twin-width in this work, and will sometimes drop the `strict' qualifier when unambiguous.

Twin-width is continuously extended to an infinite graph~$G$, i.e.
\[ \stww(G) \eqdef \sup \setst{\stww(H)}{\text{$H$ finite induced subgraph of $G$}}. \]
The graph~$G$ has \emph{finite strict twin-width} if~$\stww(G) < \infty$.

\subsubsection{Twin-width of matrices}
Twin-width extends to binary matrices, by contracting rows and columns instead of vertices.
However, in the case of matrices we only allow contraction of consecutive rows and columns.

Let~$M$ be a finite matrix.
Recall that a division of~$M$ is a partition of its row and column indices into intervals.
A division sequence for~$M$ is a sequence $(\Xc_n,\Yc_n), \dots, (\Xc_1,\Yc_1)$ of divisions of~$M$ such that
$(\Xc_n,\Yc_n)$ is the division into singletons, $(\Xc_1,\Yc_1) = (\{X\},\{Y\})$ is the trivial division,
and~$(\Xc_{i-1},\Yc_{i-1})$ is obtained from~$(\Xc_i,\Yc_i)$ by merging two parts either in~$\Xc_i$ or in~$\Yc_i$.
The \emph{width} of this division sequence is the maximum degree of
any column or row in any quotient matrix~$M / (\Xc_i,\Yc_i)$, see \cref{fig:matrix-contraction}.
The strict twin-width of~$M$, denoted by~$\stww(M)$, is the minimum width of a division sequence of~$M$.
As with graphs, it is extended to infinite matrices
by taking the supremum of strict twin-width over all finite submatrices.

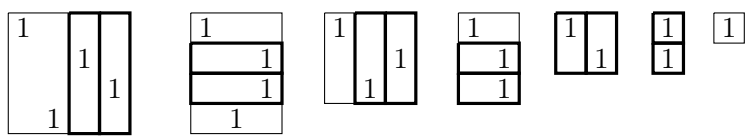
\begin{figure}
  \begin{center}
    \begin{tikzpicture}[scale=0.4]
      \def\s{4cm}

      \def\h{4}
      \def\w{4}
      \foreach \i/\j in {0/0,1/3,2/1,3/2}{
        \node (\i\j) at (\i,-\j) {1};
      }
      \draw (-0.5,0.5) -- (-0.5,-\h+0.5) -- (\w-0.5,-\h+0.5) -- (\w-0.5,0.5) -- (-0.5,0.5);
      \foreach \x in {2,3}{
        \draw[very thick] (\x-0.5,0.5) -- (\x-0.5,-\h+0.5) -- (\x+0.5,-\h+0.5) -- (\x+0.5,0.5) -- (\x-0.5,0.5);
      }

      \begin{scope}[xshift=1.5*\s]
      \def\h{4}
      \def\w{3}
      \foreach \i/\j in {0/0,1/3,2/1,2/2}{
        \node (\i\j) at (\i,-\j) {1};
      }
      \draw (-0.5,0.5) -- (-0.5,-\h+0.5) -- (\w-0.5,-\h+0.5) -- (\w-0.5,0.5) -- (-0.5,0.5);
      \foreach \y in {1,2}{
        \draw[very thick] (-0.5,-\y+0.5) -- (\w-0.5,-\y+0.5) -- (\w-0.5,-\y-0.5) -- (-0.5,-\y-0.5) -- (-0.5,-\y+0.5);
      }
      \end{scope}

      \begin{scope}[xshift=2.6*\s]
      \def\h{3}
      \def\w{3}
      \foreach \i/\j in {0/0,1/2,2/1}{
        \node (\i\j) at (\i,-\j) {1};
      }
      \draw (-0.5,0.5) -- (-0.5,-\h+0.5) -- (\w-0.5,-\h+0.5) -- (\w-0.5,0.5) -- (-0.5,0.5);
      \foreach \x in {1,2}{
        \draw[very thick] (\x-0.5,0.5) -- (\x-0.5,-\h+0.5) -- (\x+0.5,-\h+0.5) -- (\x+0.5,0.5) -- (\x-0.5,0.5);
      }
      \end{scope}

      \begin{scope}[xshift=3.7*\s]
      \def\h{3}
      \def\w{2}
      \foreach \i/\j in {0/0,1/2,1/1}{
        \node (\i\j) at (\i,-\j) {1};
      }
      \draw (-0.5,0.5) -- (-0.5,-\h+0.5) -- (\w-0.5,-\h+0.5) -- (\w-0.5,0.5) -- (-0.5,0.5);
      \foreach \y in {1,2}{
        \draw[very thick] (-0.5,-\y+0.5) -- (\w-0.5,-\y+0.5) -- (\w-0.5,-\y-0.5) -- (-0.5,-\y-0.5) -- (-0.5,-\y+0.5);
      }
      \end{scope}

      \begin{scope}[xshift=4.5*\s]
      \def\h{2}
      \def\w{2}
      \foreach \i/\j in {0/0,1/1}{
        \node (\i\j) at (\i,-\j) {1};
      }
      \draw (-0.5,0.5) -- (-0.5,-\h+0.5) -- (\w-0.5,-\h+0.5) -- (\w-0.5,0.5) -- (-0.5,0.5);
      \foreach \x in {0,1}{
        \draw[very thick] (\x-0.5,0.5) -- (\x-0.5,-\h+0.5) -- (\x+0.5,-\h+0.5) -- (\x+0.5,0.5) -- (\x-0.5,0.5);
      }
      \end{scope}

      \begin{scope}[xshift=5.3*\s]
      \def\h{2}
      \def\w{1}
      \foreach \i/\j in {0/0,0/1}{
        \node (\i\j) at (\i,-\j) {1};
      }
      \draw (-0.5,0.5) -- (-0.5,-\h+0.5) -- (\w-0.5,-\h+0.5) -- (\w-0.5,0.5) -- (-0.5,0.5);
      \foreach \y in {0,1}{
        \draw[very thick] (-0.5,-\y+0.5) -- (\w-0.5,-\y+0.5) -- (\w-0.5,-\y-0.5) -- (-0.5,-\y-0.5) -- (-0.5,-\y+0.5);
      }
      \end{scope}

      \begin{scope}[xshift=5.8*\s]
      \def\h{1}
      \def\w{1}
      \foreach \i/\j in {0/0}{
        \node (\i\j) at (\i,-\j) {1};
      }
      \draw (-0.5,0.5) -- (-0.5,-\h+0.5) -- (\w-0.5,-\h+0.5) -- (\w-0.5,0.5) -- (-0.5,0.5);
      \end{scope}
    \end{tikzpicture}
  \end{center}
  \caption{%
    A matrix contraction sequence of width 2 (i.e.\ the quotient matrices of a division sequence).
    Only 1-coefficients are represented in the matrices.
    At each step, rows or columns to be contracted are marked with thick lines.
  }
  \label{fig:matrix-contraction}
\end{figure}

\subsubsection{Grids in matrices}
For matrices, twin-width admits a dual characterisation by excluding \emph{grids}.
A $k$-grid in a matrix~$M$ is a $k \times k$-division in which all zones are non-zero,
or equivalently for which the quotient matrix has only `1' coefficients, see \cref{fig:grid}.
The \emph{grid number} of~$M$, denoted~$\gv(M)$, is the supremum of sizes of grids in~$M$.

\begin{figure}
  \begin{center}
    \begin{tikzpicture}[scale=0.4]
      \foreach \i/\j in {
        1/1,0/2,2/3,0/6,2/6,2/7,
        3/0,3/4,3/5,3/7,
        4/1,5/3,5/5,4/7,
        7/0,6/2,7/3,6/6,6/7
      }{
        \node (\i\j) at (\i,-\j) {1};
      }
      \foreach \x in {0,3,4,6,8}{
        \draw (\x-0.5,0.5) -- (\x-0.5,-7.5);
      }
      \foreach \y in {0,2,5,7,8}{
        \draw (-0.5,-\y+0.5) -- (7.5,-\y+0.5);
      }
    \end{tikzpicture}
  \end{center}
  \caption{%
    A matrix division inducing a 4-grid.
  }
  \label{fig:grid}
\end{figure}

\begin{remark}
  \label{rmk:grid-witness}
  Let~$M$ be an $X \times Y$ matrix.
  Then~$M$ contains a $k$-grid if and only if there exist $x_{i,j} \in X$, $y_{i,j} \in Y$ for $1 \le i,j \le k$ such that
  \begin{enumerate}
    \item For all indices~$i,j$, $M(x_{i,j}, y_{i,j}) = 1$
    \item For all indices~$i,i',j,j'$, if~$i < i'$ then~$x_{i,j} <_X x_{i',j'}$,
      and if~$j < j'$ then~$y_{i,j} <_Y y_{i',j'}$.
  \end{enumerate}
\end{remark}
Thus, $M$ has a $k$-grid if and only if it has a $k^2 \times k^2$-submatrix with a $k$-grid.
Hence the grid number of an infinite matrix is the supremum of the grid numbers of its finite submatrices.

Grid number and twin-width of matrices are related as follows.
\begin{theorem}[{Variant of \cite[Theorem~10]{twin-width1}}]
  \label{thm:grid-theorem-matrix}
  Fix~$d \in \Nn$. There exist functions $f,g : \Nn \to \Nn$,
  with~$f(n) = 2^{O(n)}$ and~$g(n) \sim 2n$,
  such that for any binary matrix~$M$ with degree at most~$d$,
  \begin{align*}
    \stww(M) & \le f\left(\gv(M)\right),~\text{and} \\
    \gv(M) & \le g\left(\stww(M)\right).
  \end{align*}
\end{theorem}
For finite matrices, \cref{thm:grid-theorem-matrix} is a simple generalisation
of \cite[Proposition~3.6, Theorem~4.1]{Guillemot14}
from permutation matrices to matrices with bounded degree.\footnote{
  In \cite[Theorem~4.1]{Guillemot14}, the bound on~$f$ is~$2^{O(n \log n)}$.
  The bound~$f(n) = 2^{O(n)}$ follows from improvements~\cite{Fox13,Cibulka16}
  of Marcus Tardos Theorem~\cite{MarcusT04}, at the core of this proof.
}
It is also a simplified version of \cite[Theorem~10]{twin-width1}.\footnote{
  The main differences are the following.
  \begin{enumerate}
    \item \cite[Theorem~10]{twin-width1} considers non-strict twin-width,
      and excludes \emph{mixed minors} rather than grids, allowing dense graphs and matrices.
    \item In general, excluding a $k$-mixed minor yields a division sequence of width~$f(k)$,
      but with a different order of rows and columns.
      Only when excluding grids in bounded degree matrices can we ensure a division sequence for the same order.
    \item In the general case, the bound on~$f$ is doubly exponential.
  \end{enumerate}
}

For infinite matrices, the theorem should of course be understood as stating
`$\stww(M)$ is finite if and only if~$\gv(M)$ is finite, and in that case the stated inequalities hold'.
The infinite case follows from the finite one, since twin-width and grid number of infinite matrices
are obtained by taking the supremum over finite submatrices.
In the rest of this section, we will leave implicit this kind of considerations
for infinite graphs and matrices, when they pose no particular problem.

\subsubsection{Grid number of graphs}
\label{sec:grid-number-graph}
Given a finite graph~$G$ and a partition sequence of width~$k$,
one can always find a total order~$<$ on~$V(G)$ such that the partitions of that sequence only consist of intervals.
For this order, the partition sequence immediately translates
to a division sequence of width~$k+1$ in the adjacency matrix~$\adj{<}{G}$.
This suggests the following definition: The grid number of~$G$, denoted~$\gv(G)$,
is the minimum of~$\gv(\adj{<}{G})$ over all choices of~$<$ total order on~$V(G)$.

Similarly to matrices, the grid number of an infinite graph
is the supremum of grid number over finite induced subgraphs,
but here an additional argument is needed to order the infinite graph.
\begin{lemma}[{\cite[Proposition 8.1]{twin-width2}}]
  \label{lem:grid-compactness}
  For a graph~$G$ and~$k \in \Nn$, $\gv(G) \le k$ if and only if
  for all~$H$ finite induced subgraph, $\gv(H) \le k$.
\end{lemma}
\begin{proof}
  For each finite induced subgraph~$H$ of~$G$, fix an order~$<_H$ witnessing that~$H$ has grid number at most~$k$.
  Recall that~$G$ is direct limit of its finite induced subgraphs,
  hence by \cref{lem:direct-limit-order}, there is an order~$<$ on~$G$
  such that for any finite~$X \subset V(G)$, $<$ coincides with~$<_H$ on~$X$
  for some finite induced subgraph~$H$ containing~$X$.

  We claim that~$G$ has grid number at most~$k$ with regards to order~$<$.
  Indeed, suppose on the contrary that~$G$ has a~$(k+1)$-grid in this order.
  By \cref{rmk:grid-witness}, this grid is formed by just~$(k+1)^2$ elements of~$G$.
  Choose~$H$ finite subgraph such that~$<$ and~$<_H$ coincide on these elements.
  The same elements give a~$(k+1)$-grid for~$H$ ordered by~$<_H$, a contradiction.
\end{proof}

We can now adapt \cref{thm:grid-theorem-matrix} to graphs.
\begin{theorem}[{Variant of \cite[Theorem~14]{twin-width1}}]
  \label{thm:grid-theorem-graph}
  Fix $d \in \Nn$. There exist functions $f,g : \Nn \to \Nn$,
  with $f(n) = 2^{O(n)}$ and $g(n) \sim 2n$,
  such that for any graph $G$ with degree at most $d$,
  \begin{align*}
    \stww(G) \le f\left(\gv(G)\right),~\text{and} \\
    \gv(G) \le g\left(\stww(G)\right).
  \end{align*}
\end{theorem}
\begin{proof}
  In the finite case, this is a simple variant of \cref{thm:grid-theorem-matrix},
  where we preserve the symmetry between rows and columns in the adjacency matrix
  while building the division sequence (see~\cite[Theorem~14]{twin-width1}).
  \Cref{lem:grid-compactness} is used to prove the infinite case.
\end{proof}

\subsection{Some stability results}
\subsubsection{Substitutions in bijection matrices}
Let~$A,B$ be ordered sets, and $(X_a)_{a \in A}$, $(Y_b)_{b \in B}$ families of non-empty ordered sets.
Define the sums
\[ X \eqdef \biguplus_{a \in A} X_a,~\text{and}~Y \eqdef \biguplus_{b \in B} Y_b, \]
ordered in the natural way, i.e.~$a < a'$ in~$A$ implies~$X_a < X_{a'}$ in~$X$,
while the order inside each~$X_a$ is preserved, and similarly for~$Y$.

Let $f : A \to B$ be a bijection, and for each~$a \in A$ let $g_a : X_a \to Y_{f(a)}$ also be a bijection.
Then, the \emph{substitution} of~$a$ by~$g_a$ in~$f$
is the bijection $f[a := g_a]_{a \in A} : X \to Y$ obtained by combining the~$g_a$ along~$f$, i.e.
\[ \text{for $x \in A$, $y \in X_x$} \quad f[a := g_a]_{a \in A}(x,y) \eqdef (f(x), g_x(y)). \]
In terms of matrices, we also write $\bij f [(a,f(a)) := \bij{g_a}]$ for $\bij{f[a := g_a]}$,
and call it the substitution of~$(a,f(a))$ by~$\bij{g_a}$ in~$\bij f$.

A common special case of the above is when the maps~$g_a$ are all isomorphic to some~$g : C \to D$.
Then~$X = A \times C$, $Y = B \times D$ are ordered lexicographically,
and~$f[a := g]_{a \in A}$ is the product map~$(f,g)$,
while~$\bij{(f,g)} = \bij f \otimes \bij g$ is a tensor product, see \cref{fig:substitution}.

\begin{figure}
  \begin{center}
    \begin{tikzpicture}[scale=0.35]
      \begin{scope}[yshift=2cm]
      \foreach \i/\j in {0/0,1/3,2/1,3/2}{
        \node (\i\j) at (\i,-\j) {1};
      }
      \def\h{4}
      \def\w{4}
      \draw (-0.5,0.5) -- (-0.5,-\h+0.5) -- (\w-0.5,-\h+0.5) -- (\w-0.5,0.5) -- (-0.5,0.5);
      \end{scope}

      \node (p) at (4.5,0.5) {$\otimes$};

      \begin{scope}[yshift=1cm,xshift=6cm]
      \foreach \i/\j in {0/1,1/0}{
        \node (\i\j) at (\i,-\j) {1};
      }
      \def\h{2}
      \def\w{2}
      \draw (-0.5,0.5) -- (-0.5,-\h+0.5) -- (\w-0.5,-\h+0.5) -- (\w-0.5,0.5) -- (-0.5,0.5);
      \end{scope}

      \node (p) at (8.5,0.5) {$=$};

      \begin{scope}[yshift=4cm,xshift=10cm]
      \foreach \i/\j in {0/0,1/3,2/1,3/2}{
        \node (\i\j0) at (2*\i,-2*\j-1) {1};
        \node (\i\j1) at (2*\i+1,-2*\j) {1};
        \draw[thin,gray] (2*\i-0.5,-2*\j+0.5) -- (2*\i-0.5,-2*\j-1.5) -- 
                         (2*\i+1.5,-2*\j-1.5) -- (2*\i+1.5,-2*\j+0.5) -- (2*\i-0.5,-2*\j+0.5);
      }
      \def\h{8}
      \def\w{8}
      \draw (-0.5,0.5) -- (-0.5,-\h+0.5) -- (\w-0.5,-\h+0.5) -- (\w-0.5,0.5) -- (-0.5,0.5);
      \end{scope}
    \end{tikzpicture}
  \end{center}
  \caption{The tensor product corresponds to substituting each 1-coefficient of the left matrix by the right matrix.}
  \label{fig:substitution}
\end{figure}

\begin{lemma}
  \label{lem:tww-substitution}
  Twin-width is invariant by substitution in bijection matrices
  \[ \stww(\bij{f[a := g_a]_{a \in A}}) = \max\left(\stww(\bij f), \max_{a \in A} \stww(\bij{g_a})\right). \]
\end{lemma}
\begin{proof}
  In the finite case, this is \cite[Proposition~3.4]{Guillemot14}:
  in~$\bij{f[a := g_a]}$, one begins by regrouping each block~$\bij{g_a}$ in a single part
  using a division sequence for~$\bij{g_a}$.
  After this, the remaining quotient matrix is exactly~$\bij f$.

  The infinite case reduces to the finite one by remarking that any finite submatrix of~$\bij{f[a := g_a]}$
  is obtained by substituting finite submatrices of some~$\bij{g_a}$ into a finite submatrix of~$\bij f$.
\end{proof}

Because twin-width is defined through finite submatrices,
the above argument still holds for infinitely nested substitutions.
To avoid overwhelming notations, we will only use it for infinite tensor products.

\begin{lemma}
  \label{lem:tww-tensor-product}
  Let~$(I,<_I)$ be a well-ordered set of indices, and for each~$i \in I$
  consider $(X_i,<_{X_i})$, $(Y_i,<_{Y_i})$ two totally ordered sets, and $f_i : X_i \to Y_i$ a bijection.
  Consider the products $X \eqdef \prod_{i \in I} X_i$, $Y \eqdef \prod_{i \in I} Y_i$.
  They are totally ordered by lexicographic order, i.e.\ for~$x \neq x'$
  \[
    x <_X x' \quad \iff \quad x_i <_{X_i} x'_i
    \qquad \text{where } i \eqdef \min_{j \in I} x_j \neq x'_j.
  \]
  Let $f : X \to Y$ be the product of all~$f_i$. Then
  \[ \stww(\bij{f}) = \sup_{i \in I} \stww(\bij{f_i}). \]
\end{lemma}
\begin{proof}
  Let~$M'$ be a finite submatrix of~$\bij{f}$ on rows~$X'$ and columns~$Y'$.
  It suffices to bound~$\stww(M')$ by the right-hand side of the lemma.
  For each pair of distinct elements of~$X'$ or~$Y'$, consider the minimal index in~$I$ on which these elements differ,
  and let~$J$ be the (finite) collection of all these indices.
  Define~$X_J = \prod_{j \in J} X_j$ and~$Y_J = \prod_{j \in J} Y_j$,
  equipped with the projections~$\pi_X : X \to X_J$ and~$\pi_Y : Y \to Y_J$.
  Consider the bijection $f_J : X_J \to Y_J$ obtained by product of~$(f_j)_{j \in J}$.
  By construction, $\pi_X$ is injective from~$X'$ to~$X_J$, and similarly for~$\pi_Y$.
  Furthermore, it is clear that if~$y = f(x)$, then~$\pi_Y(y) = f_J(\pi_X(x))$.
  Finally, if~$X_J$ and~$Y_J$ are again ordered lexicographically, $\pi_X$ and~$\pi_Y$ are non-decreasing.
  From this it follows that $\stww(M') \le \stww(\bij{f_J})$,
  and since~$\bij{f_J}$ is a finite tensor product of~$(\bij{f_j})_{j \in J}$,
  we have $\stww(\bij{f_J}) \le \sup_{i \in I} \stww(\bij{f_i})$ as desired.
\end{proof}

\subsubsection{Superposition}
Given~$M,N$ two~$X \times Y$ matrices, define their superposition~$M \lor N$ as the~$X \times Y$ matrix with
\[ (M \lor N)(x,y) \eqdef \max\left(M(x,y),N(x,y)\right). \]
A variant of the next lemma was recently observed \cite[Lemma 18]{twin-width8}.

\begin{lemma}
  \label{lem:grid-ramsey}
  If~$M_1,\dots,M_r$ are~$X \times Y$ matrices with~$\gv(M_i) < k$, then
  \[ \gv\left(\textstyle{\bigvee_{i=1}^r M_i} \right) < r^{rk}. \]
\end{lemma}
\begin{proof}
  A classical Ramsey theory argument gives the following.
  \begin{claim}
    For any~$k \in \Nn$, fix~$\ell = r^{rk}$.
    Then for any~$\ell \times \ell$ matrix~$M$ whose coefficients are colors in~$\{1,\dots,r\}$,
    there exists a~$k \times k$ submatrix of~$M$ consisting of only one color.
  \end{claim}

  Now assume that~$\bigvee_{i=1}^r M_i$ contains a~$\ell$-grid.
  Consider the quotient matrix corresponding to this grid,
  and color a zone~$Z$ with the minimum~$i$ such that~$Z$ is non-zero in~$M_i$.
  There must be such an~$i$ for each zone since all zones are non-zero in the superposition of all~$M_i$.
  The claim then gives a~$k \times k$ submatrix with a single color~$i$, which induces a~$k$-grid in~$M_i$.
\end{proof}

\subsubsection{Bounded quotients}
\begin{lemma}
  \label{lem:graph-bounded-quotient}
  Let~$G$ be a graph and~$\Pc$ a partition of~$V(G)$ such that every part~$P \in \Pc$ has size~$\card{P} \le k$.
  Then
  \[ \stww(G) \le k \stww(G / \Pc). \]
\end{lemma}
\begin{proof}
  Given a partition sequence for~$G / \Pc$ of width~$t$,
  it can be extended to a partition sequence for~$G$ by
  first going from the partition in singletons of~$G$ to~$\Pc$ in an arbitrary manner.
  During this process, the quotient graphs are of the form~$G / \Pc'$,
  where every part of~$\Pc'$ is contained in some part of~$\Pc$.
  Since every~$P \in \Pc$ has size at most~$k$,
  it contains at most~$k$ distinct parts of~$\Pc'$,
  which clearly implies that~$\Delta(G / \Pc') \le k \Delta(G / \Pc)$.
  It follows that
  \[ \stww(G) \le \max\left( \stww(G / \Pc) , k \Delta(G / \Pc) \right). \]
  and the right-hand side is bounded by~$k \stww(G / \Pc)$.
\end{proof}

\subsubsection{First-order interpretations}
\label{sec:interpretation}
If~$\Cc$ is a class of finite graphs of bounded twin-width,
then any first-order transduction of~$\Cc$ also has bounded twin-width~\cite[Theorem~39]{twin-width1}.
However, this result only holds with non-strict twin-width,
because first-order transductions can produce dense graphs from graphs of bounded degree.
In our setting, we will only consider two special cases of interpretations which preserve bounded degree:
powers of graphs, and composition of permutation matrices.
We give simpler ad hoc proofs for these cases.

Let~$G$ be a graph. The $k$-th power of~$G$, denoted by~$G^{(k)}$, is the graph with the same vertices as~$G$,
and such that~$x \neq y$ are adjacent in~$G^{(k)}$ if and only if~$d_G(x,y) \le k$.
Remark that
\begin{equation}
  \Delta(G^{(k)}) \le (\Delta(G))^k.
  \label{eq:degree-power}
\end{equation}
\begin{lemma}
  \label{lem:power-graph}
  For any~$k \in \Nn$, and any graph~$G$,
  \[ \stww(G^{(k)}) \le \stww(G)^k. \]
\end{lemma}
\begin{proof}
  We prove the finite case.
  If~$\Pc$ is a partition of~$V(G)$, one verifies that
  \begin{equation}
    G^{(k)} / \Pc \quad \text{is a subgraph of} \quad (G / \Pc)^{(k)}.
    \label{eq:power-quotient-subgraph}
  \end{equation}
  Indeed, for~$X,Y \in \Pc$, there is an edge~$XY$ in~$G^{(k)} / \Pc$ if and only if
  there are~$x \in X$, $y \in Y$ and a path of length at most~$k$ in~$G$ from~$x$ to~$y$.
  When that is the case, the parts containing the successive vertices of this path
  yield a path from~$X$ to~$Y$ of length at most~$k$ in~$G / \Pc$, see \cref{fig:quotient-square}.
  \begin{figure}
    \begin{center}
      \begin{tikzpicture}
      \tikzstyle{every node}=[normalnode,label distance=0.3cm]
      \node[label=below:$a_1$] (a1) at (0,0) {};
      \node[label=above:$a_2$] (a2) at (0,1) {};
      \node[label=below:$d$] (d) at (1.5,0) {};
      \node[label=above:$e$] (e) at (1.5,1) {};
      \node[label=right:$B$,label=below:$b$] (B) at (3,0) {};
      \node[label=right:$C$,label=above:$c$] (C) at (3,1) {};

      \draw (a1) -- (d) -- (a2);
      \draw (B) -- (d);
      \draw (C) -- (e);
      \draw (0,0.5) node [draw=none, fill=none, label=left:$A$] {} ellipse (0.25 and 0.8);
      \draw (1.5,0.5) ellipse (0.25 and 0.8);
      \draw (B) circle (0.3);
      \draw (C) circle (0.3);
      \end{tikzpicture}
    \end{center}
    \caption{%
      A graph~$G$ with a partition~$\Pc$.
      In both~$G^{(2)}/\Pc$ and~$(G/\Pc)^{(2)}$, parts~$A$ and~$B$ are adjacent because of the path~$a_1,d,b$.
      In~$(G/\Pc)^{(2)}$ only, $A$ and~$C$ are adjacent thanks to the path~$a_1,d,e,c$,
      which can `jump' from~$d$ to~$e$ since they are in the same part of~$\Pc$.
    }
    \label{fig:quotient-square}
  \end{figure}
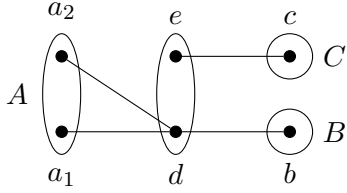

  Now if~$\Pc_n,\dots,\Pc_1$ is a partition sequence of width~$t$ for~$G$,
  then~$(G / \Pc_i)^{(k)}$ has degree bounded by~$t^k$ by~\cref{eq:degree-power}, hence so does~$G^{(k)} / \Pc_i$.
  Thus~$\Pc_n,\dots,\Pc_1$ is also a partition sequence of width at most~$t^k$ for~$G^{(k)}$.
\end{proof}

\begin{lemma}
  \label{lem:perm-composition}
  If $\sigma : X \to Y$, $\tau : Y \to Z$ are bijections between totally ordered sets
  with $\stww(\bij{\sigma}),\stww(\bij{\tau}) \le t$, then
  \[ \stww(\bij{\tau \circ \sigma}) = 2^{O(t)}. \]
\end{lemma}
\begin{proof}
  By \cref{thm:grid-theorem-matrix}, assume that $\gv(\bij\sigma), \gv(\bij\tau) \le t'$ with~$t' = O(t)$.
  We order~$X \uplus Z$ with~$X < Z$, and consider the $(X \uplus Z) \times Y$ matrix~$M$
  consisting of~$\bij\sigma$ and~$\bij{\tau^{-1}}$ put side by side
  (remark that~$\bij{\tau^{-1}}$ is the transpose of~$\bij\tau$).
  Clearly, $\gv(M) \le 2t'+1$, hence by \cref{thm:grid-theorem-matrix},
  there is a division sequence $(\Rc_n,\Cc_n), \dots, (\Rc_1,\Cc_1)$ of~$M$ of width~$2^{O(t)}$.

  Let us construct a division sequence $(\Xc_n,\Zc_n), \dots, (\Xc_1,\Zc_1)$ for~$\bij{\tau \circ \sigma}$.
  The partition~$\Xc_i$ of~$X$ into intervals is obtained by restricting~$\Rc_i$ to~$X$, i.e.
  \[ \Xc_i \eqdef \setst{R \cap X}{R \in \Rc_i \land R \cap X \neq \emptyset}. \]
  The partition~$\Zc_i$ of~$Z$ is defined similarly.
  It is easy to verify that for any~$i$, either $(\Xc_i,\Zc_i) = (\Xc_{i+1},\Zc_{i+1})$,
  or~$(\Xc_i,\Zc_i)$ is obtained by merging two parts in either~$\Xc_{i+1}$ or~$\Zc_{i+1}$.
  Hence, up to some trivial cleanup, this is a division sequence for~$\bij{\tau \circ \sigma}$.

  We need to bound the degree of the quotient matrices $\bij{\tau \circ \sigma} / (\Xc_i,\Zc_i)$.
  Let~$X' \in \Xc_i$ be a part, obtained as~$X' = X \cap R$ for some~$R \in \Rc_i$.
  Say that~$Y' \in \Cc_i$ is a neighbour of~$X'$ if there exists~$x \in X'$ with~$f(x) \in Y'$.
  The number of neighbours of~$X'$ in~$\Cc_i$ is bounded by the degree of~$R$ in~$M / (\Rc_i,\Cc_i)$.
  Similarly, for~$Y' \in \Cc_i$, say that~$Z' \in \Zc_i$ is a neighbour of~$Y'$
  if there exists~$y \in Y'$ with~$g(y) \in Z'$.
  Again, the number of neighbours of~$Y'$ in~$\Zc_i$ is bounded by the degree of~$Y'$ in~$M / (\Rc_i,\Cc_i)$.
  Now, if~$Z' \in \Zc_i$ is such that there exists~$x \in X'$ with~$g(f(x)) \in Z'$,
  then~$Z'$ is a neighbour in~$\Zc_i$ of a neighbour in~$\Cc_i$ of~$X'$.
  It follows that the degree of~$X'$ in~$\bij{\tau \circ \sigma} / (\Xc_i,\Zc_i)$
  is bounded by the square of the maximum degree in~$M / (\Rc_i,\Cc_i)$, which is still~$2^{O(t)}$.
  The same holds for the degree of a part~$Z' \in \Zc_i$.

  Thus we have constructed a division sequence for~$\bij{\tau \circ \sigma}$ of width~$2^{O(t)}$.
\end{proof}

The former proof could be adapted to show that the composition of~$k$
bijection matrices of twin-width~$t$ has twin-width~$2^{O(tk)}$,
instead of the tower of exponentials obtained repetitively applying the former lemma.

\section{Twin-width as a geometric group invariant}
\label{sec:geometric-tww}
In \cite[Lemma~8.2]{twin-width2}, we already proved finite twin-width is a group invariant.
This can be stated as a more general result on regular embeddings.
\begin{lemma}
  \label{lem:tww-regular-embedding}
  If $f : H \to G$ is a $\lambda$-regular embedding, then
  \[ \stww(H) \le \lambda \stww(G)^\lambda. \]
\end{lemma}
\begin{proof}
  Let $\Pc = \setst{f^{-1}(v)}{v \in V(G)}$ be the preimage partition of~$V(H)$.
  Since~$f$ is $\lambda$-regular, parts of~$\Pc$ have size at most~$\lambda$.
  One can see~$f$ as a map $f : H / \Pc \to G$, which is injective and $\lambda$-Lipschitz.
  This implies that~$H / \Pc$ is isomorphic to a subgraph of~$G^{(\lambda)}$.
  Then, by \cref{lem:graph-bounded-quotient,lem:power-graph},
  \[ \stww(H) \le \lambda \stww(H / \Pc) \le \lambda \stww(G)^\lambda. \]
\end{proof}

Since quasi-isometries between graphs of bounded degree are regular embeddings,
it follows that having finite twin-width is invariant under quasi-isometry for graphs of bounded degree.
In particular, whether or not a Cayley graph~$\cay(\Gamma,S)$ has finite twin-width
depends only on the group~$\Gamma$, not on the generating set~$S$.
This defines the notion of finite twin-width for finitely generated groups.

Let us give some simple examples.
\begin{enumerate}
  \item \label{item:finite} Finite groups trivially have finite twin-width.
  \item \label{item:free} It is easy to see that the strict twin-width of a tree is exactly its maximum degree.
    Thus, the free group with~$k$ generators has finite twin-width, since it admits the $k$-regular tree as Cayley graph.
  \item \label{item:product} A cartesian product of groups corresponds to the cartesian product of Cayley graphs,
    which is known to preserve finite strict twin-width (see~\cite[Theorem~7.2]{twin-width2}).
  \item By~\eqref{item:finite} and~\eqref{item:free}, all cyclic groups have finite twin-width,
    and thus by~\eqref{item:product}, all finitely generated abelian groups have finite twin-width.
\end{enumerate}

Finite twin-width being closed under quasi-isometry,
it is natural to wonder if other well-known properties of groups preserved by quasi-isometry imply finite twin-width.
Hyperbolicity in the sense of Gromov~\cite{gromov1987hyperbolic}
is arguably the most important quasi-isometric invariant of groups.
By a remarkable result of Buyalo, Dranishnikov, and Schroeder,
hyperbolic groups---and more generally, hyperbolic graphs of bounded degree---%
are quasi-isometric to finite Cartesian products of trees of bounded degree~\cite{buyalo2007embedding}.
Since trees of bounded degree have finite twin-width, and Cartesian products preserve it, we obtain the following.

\begin{proposition}
  \label{lem:tww-hyperbolic}
  Gromov-hyperbolic groups and graphs of bounded degree have finite twin-width.
\end{proposition}

\section{Matrix characterisation}
\label{sec:utww}
In this section, we give a second characterisation of twin-width through matrices.
This allows to generalise it beyond finitely generated groups,
and suggests a natural strengthening, called \emph{uniform twin-width}.

\subsection{Twin-Width of group actions}
Let~$\Gamma$ be a group acting on~$X$ by $\phi : \Gamma \to \Perm_X$,
where~$\Perm_X$ is the group of permutations on~$X$.
Given a total order~$<$ on~$X$ and~$g \in \Gamma$,
the (strict) twin-width of~$g$ acting on~$(X,<)$ is
\[ \stww_{(X,<)}(\phi(g)) \eqdef \stww\left(\bij{\phi(g)}\right), \]
where the $X \times X$ bijection matrix~$\bij{\phi(g)}$ is ordered according to~$<$.

The action of~$\Gamma$ through~$\phi$ is said to have finite twin-width with regards to~$<$
if~$\stww_{(X,<)}(\phi(g))$ is finite for every~$g \in \Gamma$---but possibly arbitrarily large.
This action is said to have finite twin-width
it has finite twin-width with regards to some total order on~$X$.
Finite grid number for a group action is defined similarly.

\begin{lemma}
  \label{lem:grid-theorem-group}
  For a finitely generated group~$\Gamma$, the following are equivalent.
  \begin{enumerate}
    \item \label{item:cayley-tww} Some (equivalently, every) Cayley graph of~$\Gamma$
      has finite twin-width (equivalently, finite grid number).
    \item \label{item:matrix-tww} The action of $\Gamma$ on itself by right product
      has finite twin-width (equivalently, finite grid number).
  \end{enumerate}
\end{lemma}
\begin{proof}
  In the statement, equivalences of finite twin-width and finite grid number
  are from \cref{thm:grid-theorem-matrix,thm:grid-theorem-graph}.
  For~$x \in \Gamma$ and~$<$ a total order on~$\Gamma$,
  we denote by~$M^<_x$ the matrix of the action of~$x$ on~$(\Gamma,<)$ by right product.

  Suppose that the action of~$\Gamma$ on itself has finite grid number,
  witnessed by some total order~$<$ on~$\Gamma$.
  Let us prove that~$\cay(\Gamma,S)$ has finite grid number for an arbitrary finite generating set~$S$.
  Let~$M$ be the adjacency matrix of~$\cay(\Gamma,S)$ with regards to~$<$.
  Then~$M$ is the superposition
  \[ M = \biglor_{s \in S \cup S^{-1}} M^<_s \]
  Any single~$M^<_x$ has finite grid number by hypothesis, and~$S \cup S^{-1}$ is finite,
  thus it follows from \cref{lem:grid-ramsey} that~$M$ has finite grid number.

  Conversely, assume that~$\cay(\Gamma,S)$ has finite grid number for some finite generating set~$S$.
  This is witnessed by a total order~$<$ on~$\Gamma$ such that
  the adjacency matrix~$M$ of~$\cay(\Gamma,S)$ with regards to~$<$ has finite twin-width.
  For any~$s \in S \cup S^{-1}$, one easily verifies that~$M^<_s$ is contained in~$M$,
  in the sense that~$M^<_s(x,y) = 1$ implies~$M(x,y) = 1$.
  It follows that~$M^<_s$ has finite twin-width.
  If~$x \in \Gamma$ is now an arbitrary element, $x$ can be decomposed as some product~$x = s_1 \dots s_n$
  with~$s_i \in S \cup S^{-1}$, and the action of~$x$ is the composition of the actions of~$s_1,\dots,s_n$.
  Since all~$M^<_{s_i}$ have finite twin-width,
  it follows from \cref{lem:perm-composition} that~$M^<_x$ has finite twin-width.
\end{proof}

Thus we can take as alternative definition that~$\Gamma$ has finite twin-width
if its action on itself by right product has finite twin-width.
This definition immediately generalises to non-finitely generated groups.

\subsection{Uniform twin-width}
There is a natural strengthening of the former characterisation.
Consider an action of~$\Gamma$ on~$X$ given by~$\phi : \Gamma \to \Perm_X$.
Given a total order~$<$ on~$X$, the \emph{uniform twin-width} of~$\phi$ with regards to~$<$ is
\[ \utww_{(X,<)}(\phi) \eqdef \sup_{g \in \Gamma} \stww_{(X,<)}(\phi(g)). \]
Then, the uniform twin-width of~$\phi$, denoted by~$\utww(\phi)$,
is the minimum of~$\utww_{(X,<)}(\phi)$ over all choices of~$<$ total order on~$X$.
The uniform twin-width of a group~$\Gamma$, denoted~$\utww(\Gamma)$,
is defined as the uniform twin-width of its action on itself by right product.

Unlike twin-width, the uniform twin-width of a group or group action is a~well defined number (or infinity).
For some fixed ordering, having finite uniform twin-width
requires a uniform bound on the twin-width of the actions of each element,
whereas having finite twin-width only requires the twin-width of each of these actions to be finite.
Thus finite uniform twin-width implies finite twin-width, while the converse is not a priori true.

\subsection{Examples}
\label{sec:utww-examples}
Let us refine the examples of \cref{sec:geometric-tww}.
A \emph{right-invariant} order on a group~$\Gamma$ is a total order~$<$
such that for all~$x,y,z \in \Gamma$, $x < y$ implies~$xz < yz$.
A group is \emph{(right-)orderable} if it admits a right-invariant total order.
Groups of finite twin-width can be seen as a vast generalisation of orderable groups.
Indeed, if~$<$ is a right-invariant order on~$\Gamma$,
then the action of any~$x \in \Gamma$ by right product is monotone,
hence its matrix has strict twin-width~2 (the minimum possible value for non-trivial matrices).
\begin{proposition}
  \label{prop:ordered-tww}
  Right-orderable groups have uniform twin-width~2.
\end{proposition}

Finitely generated free groups are known to be orderable, hence have uniform twin-width~2.
It is easy to see that finite cyclic groups, with the natural order, also have uniform twin-width~2.
Using \cref{lem:tww-substitution}, we obtain that Cartesian product preserves uniform twin-width, i.e.
\[ \utww(\Gamma_1 \times \Gamma_2) = \max(\utww(\Gamma_1), \utww(\Gamma_2)). \]
This is witnessed by a lexicographic order on~$\Gamma_1 \times \Gamma_2$.
Therefore all finitely generated abelian groups have uniform twin-width~2.

\section{Groups with finite twin-width}
\label{sec:finite-tww}
In this section, we prove that uniform twin-width is preserved by a number of group constructions,
and use these results to provide examples of groups with finite twin-width.
Here, unlike the rest of the paper, we use~$G,H$ etc.\ to denote groups and not graphs.

Let us start with the trivial remark that twin-width is monotone under taking subgroups.
It is used for lower bounds in results of this section.
\begin{lemma}
  If~$G$ is a group, $H \le G$ is a subgroup, and~$G$ has finite twin-width, then so does~$H$, and
  \[ \utww(H) \le \utww(G). \]
  \label{lem:subgroup}
\end{lemma}

\subsection{Limits}
The characterisation of twin-width through matrices allows generalisation to infinitely generated groups.
However, uniform twin-width can be reduced to the case of finitely generated groups,
because it is stable by direct limits.

\begin{lemma}
  \label{lem:group-limit}
  If~$G$ is direct limit of~$(G_i)_{i \in I}$, then
  \[ \utww(G) = \sup_{i \in I} \utww(G_i). \]
\end{lemma}
\begin{proof}
  For each~$G_i$, let~$<_i$ be an order witnessing that~$\utww(G_i) \le k$.
  By \cref{lem:direct-limit-order}, there is an order~$<$ on~$G$ such that for any finite~$X \subset G$,
  the orders~$<$ and~$<_i$ coincide on~$X$ for some~$i$.

  For~$x \in G$, let~$M_x$ be the matrix of~$x$ acting on~$(G,<)$ by right product.
  We want to show $\stww(M_x) \le k$, i.e.\ that for every finite submatrix $M$ of $M_x$, $\stww(M) \le k$.
  Let~$X$ contain the row and column indices of~$M$, plus~$x$.
  Since~$X$ is finite, there is some~$i$ with~$X \subset G_i$ and such that~$<$ and~$<_i$ coincide on~$X$.
  Then~$M$ is a submatrix of the matrix of~$x$ acting on~$(G_i,<_i)$ by right product,
  hence~$\stww(M) \le k$, which implies the result.
\end{proof}

\begin{corollary}
  \label{cor:utww-infinite-gen}
  Let~$G$ be a group and let~$\mathcal{H}$ be the collection of finitely generated subgroups of~$G$.
  Then
  \[ \utww(G) = \sup_{H \in \mathcal{H}} \utww(H). \]
\end{corollary}
\begin{proof}
  The group~$G$ is direct limit of~$\mathcal{H}$.
\end{proof}

In \cref{sec:utww-examples}, we already proved that finitely generated abelian groups have uniform twin-width~2.
Applying \cref{cor:utww-infinite-gen}, we obtain
\begin{proposition}
  \label{prop:abelian-tww}
  Abelian groups have uniform twin-width~2.
\end{proposition}

\subsection{Products, subgroups, and quotients}
In \cref{sec:geometric-tww,sec:utww-examples}, we made the remark that
(uniform) twin-width is stable under Cartesian product.
We will now give a significant generalisation of this result.

Let~$G$ be a group and~$H \le G$.
We write~$G / H$ for the right quotient, i.e.\ the set of right cosets~$\setst{Hg}{g \in G}$.
Then, $G$ acts on~$G/H$ by right product,
and we abusively talk about the (uniform) twin-width of~$G/H$ when meaning the (uniform) twin-width of this action.
When~$H$ is a normal subgroup, the (uniform) twin-width of~$G/H$ as group
indeed coincides with that of~$G$ acting on~$G/H$, justifying this convention.
\begin{lemma}
  \label{lem:extension-gen}
  For any group~$G$ and~$H \le G$, the following statements hold.
  \begin{enumerate}
    \item $\utww(G) \le \max\left( \utww(H), \utww(G/H) \right)$.
    \item If~$H$ has finite \emph{uniform} twin-width and~$G/H$ has finite twin-width,
      then~$G$ has finite twin-width.
    \item If~$H$ has finite twin-width and~$G/H$ is finite, $G$ has finite twin-width.
  \end{enumerate}
\end{lemma}
Before proving it, we list some direct, useful corollaries, with more natural statements.
\begin{corollary}
  \label{cor:extension}
  If~$G$ is an extension of~$Q$ by~$H$, i.e.~$H \lhd G$ and~$Q \iso G/H$
  \[ \utww(G) \le \max(\utww(H), \utww(Q)). \]
\end{corollary}
\begin{corollary}
  \label{cor:semidirect-product}
  If~$G \rtimes H$ is any semi-direct product, then
  \[ \utww(G \rtimes H) = \max(\utww(G), \utww(H)). \]
\end{corollary}
\begin{corollary}
  If~$G$ is a group with~$H \le G$ a subgroup of finite index $[G : H] = k$,
  then~$G$ has finite twin-width if and only if~$H$ does, and
  \[ \utww(H) \le \utww(G) \le \max(\utww(H),k). \]
  \label{cor:finite-index}
\end{corollary}

\begin{proof}[Proof of \cref{lem:extension-gen}]
  For~$x \in G$, let~$\bar x \eqdef Hx$ be its equivalence class in~$G/H$.
  We choose a transversal~$T \subset G$ of~$G/H$, meaning that each class of~$G/H$
  can be uniquely written as~$\bar t$ for some~$t \in T$.
  Suppose given total orders~$<_H$, $<_{G/H}$ on~$H$ and~$G/H$
  witnessing their twin-width or uniform twin-width.
  We define a total order~$<$ on~$G$ as follows.
  \begin{itemize}
    \item The order between cosets is defined by
      \[ \bar x <_{G/H} \bar y \quad \Longrightarrow \quad x < y. \]
    \item For~$t \in T$, the order inside the coset~$\bar t$ is defined by
      \[ \forall x,y \in \bar t, \quad x < y \quad \iff \quad x t^{-1} <_H y t^{-1}. \]
  \end{itemize}

  For~$a \in G$, let~$M^G_a$ be the matrix of~$a$ acting on~$G$ by right product with regards to~$<$.
  Similarly, $M^{G/H}_a$ is the matrix of~$a$ acting on~$(G/H,<_{G/H})$,
  and for~$b \in H$, $M^H_b$ is the matrix of~$b$ acting on~$(H,<_H)$.
  Fix some~$a \in G$.

  \begin{claim}
    Suppose that~$\stww(M^{G/H}_a) \le k$,
    and that for any~$t_1,t_2 \in T$, for~$r \eqdef t_1 a t_2^{-1} \in H$, we have~$\stww(M^H_r) \le k$.
    Then
    \[ \stww(M^G_a) \le k. \]
  \end{claim}
  \begin{proof}
    Let~$t_1,t_2 \in T$ be such that~$t_1 a t_2^{-1} \in H$, or equivalently,
    right multiplication by~$a$ maps the coset~$\bar t_1 = H t_1$ to~$\bar t_2 = H t_2$.
    In that case, define~$P^{t_1,t_2}_a$ to be the matrix of this map,
    where~$H t_1$ and~$H t_2$ are ordered as in~$G$.
    It is easy to verify that~$M^G_a$ is obtained from~$M^{G/H}_a$
    by substituting the~`1' at position~$(\bar t_1, \bar t_2)$ by~$P^{t_1,t_2}_a$
    for all such choices of~$t_1,t_2$, see \cref{fig:group-extension-matrix}.
    Thus, by \cref{lem:tww-substitution}, it suffices to prove~$\stww(P^{t_1,t_2}_a) \le k$.

    \begin{figure}
      \begin{center}
        \begin{tikzpicture}
          \draw (0,-4.5) -- (0,0) -- (4.5,0);

          \node (s1) at (1.5,0.3) {$H s_1$};
          \node (t1) at (3.5,0.3) {$H t_1$};
          \node (t2) at (-0.5,-1.5) {$H t_2$};
          \node (s2) at (-0.5,-3) {$H s_2$};

          \draw [thin] (1,-2.5) -- (1,-3.5) -- (2,-3.5) -- (2,-2.5) -- (1,-2.5);
          \node (p1) at (1.5,-3) {$P^{s_1,s_2}_a$};
          \draw[dotted] (1,0) -- (1,-2.5);
          \draw[dotted] (2,0) -- (2,-2.5);
          \draw[dotted] (0,-2.5) -- (1,-2.5);
          \draw[dotted] (0,-3.5) -- (1,-3.5);

          \draw [thin] (3,-1) -- (3,-2) -- (4,-2) -- (4,-1) -- (3,-1);
          \node (p2) at (3.5,-1.5) {$P^{t_1,t_2}_a$};
          \draw[dotted] (3,0) -- (3,-1);
          \draw[dotted] (4,0) -- (4,-1);
          \draw[dotted] (0,-1) -- (3,-1);
          \draw[dotted] (0,-2) -- (3,-2);

          \draw [thin, dashed] (4,-1) -- (6.5,0);
          \draw [thin, dashed] (4,-2) -- (6.5,-4);

          \begin{scope}[xshift=6.5cm]
          \draw (0,-4) -- (0,0) -- (4,0);

          \node (x1) at (1,0.3) {$x_1 t_1$};
          \node (x2) at (3,0.3) {$x_2 t_1$};

          \node (y1) at (-0.5,-1) {$y_1 t_2$};
          \node (y2) at (-0.5,-2.5) {$y_2 t_2$};

          \node (xy1) at (1,-1) {$1$};
          \node (xy2) at (3,-2.5) {$1$};

          \draw[dotted] (1,0) -- (1,-0.7);
          \draw[dotted] (3,0) -- (3,-2.2);
          \draw[dotted] (0,-1) -- (0.7,-1);
          \draw[dotted] (0,-2.5) -- (2.7,-2.5);
          \end{scope}
        \end{tikzpicture}
      \end{center}
      \caption{%
        Example of shape of~$M^G_a$. Here, $s_i,t_i$ are in~$T$ with $\bar s_1 <_{G/H} \bar t_1$,
        $\bar t_2 <_{G/H} \bar s_2$, and $t_1 a t_2^{-1}, s_1 a s_2^{-1} \in H$.
        Furthermore, $x_1 <_H x_2$ and~$y_1 <_H y_2$ are in~$H$,
        and satisfy~$(x_i t_1) a = y_i t_2$, or equivalently $x_i r = y_i$ with $r = t_1 a t_2^{-1}$.
        The general shape of $M^G_a$ (left matrix) corresponds to $M^{G/H}_a$,
        while the block $P^{t_1,t_2}_a$ (right matrix) is equal to $M^H_r$.
      }
      \label{fig:group-extension-matrix}
    \end{figure}
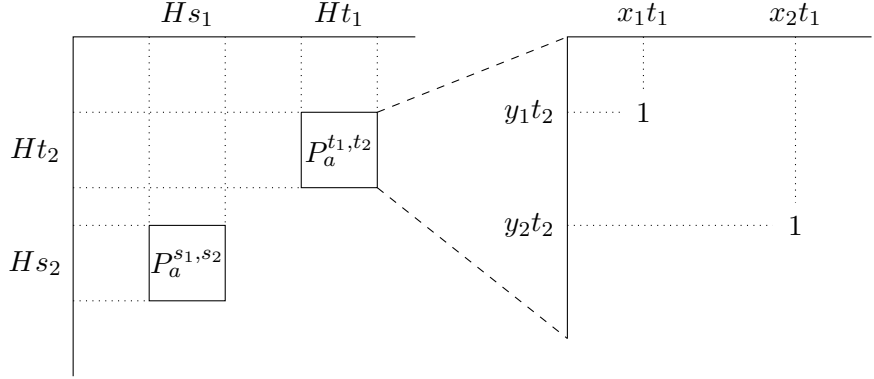

    Define~$r \eqdef t_1 a t_2^{-1} \in H$.
    We claim that~$P^{t_1,t_2}_a$ is isomorphic to~$M^H_r$.
    Indeed, right product by~$t_1$ (resp.~$t_2$) is an increasing bijection from~$H$ to~$H t_1$ (resp.~$H t_2$),
    i.e.\ from the set of column indices (resp.\ row indices) of~$M^H_r$ to that of~$P^{t_1,t_2}_a$.
    Furthermore, these bijections transform right product by~$a$ into right product by~$r$,
    in that for all~$x \in H$
    \begin{equation}
      (x t_1) a = x (t_1 a t_2^{-1}) t_2 = (x r) t_2.
    \end{equation}
    Thus~$P^{t_1,t_2}_a \iso M^H_r$, and~$\stww(P^{t_1,t_2}_a) \le k$ follows from the hypotheses.
  \end{proof}

  We now prove the three points of the lemma using the claim.
  \begin{enumerate}
    \item is immediate.
    \item Given a uniform bound on~$\stww(M^H_r)$ for all~$r \in H$, and a bound on~$\stww(M^{G/H}_a)$,
      the claim gives a bound on~$\stww(M^G_a)$.
      Remark that since the bound on~$\stww(M^H_r)$ is used for~$r = t_1 a t_2^{-1}$
      for all appropriate choices of~$t_1,t_2$, the uniform bound is a priori required.
    \item If~$G/H$ is finite however, there are only finitely many~$r = t_1 a t_2^{-1}$
      for which the bound on~$\stww(M^H_r)$ is used (for~$a$ fixed),
      hence the uniform bound in the previous argument is no longer required.
  \end{enumerate}
\end{proof}

With the previous results, we can give interesting examples of groups with finite twin-width.
The class of \emph{solvable groups} can be defined inductively as follows.
\begin{itemize}
  \item The trivial group is solvable.
  \item If~$H \lhd G$, $H$ is solvable, and~$G/H$ is abelian, then~$G$ is solvable.
\end{itemize}
Solvable groups notably contain nilpotent groups.
\begin{proposition}
  \label{prop:solvable-tww}
  Solvable groups have uniform twin-width 2.
\end{proposition}
\begin{proof}
  Immediate by \cref{prop:abelian-tww,cor:extension}.
\end{proof}

If~$\Gamma$ is a group finitely generated by~$S$, its \emph{growth function} is
\[ g_{(\Gamma,S)} (r) \eqdef \card{B(1_\Gamma,r)}, \]
where~$B(1_\Gamma,r)$ is the ball of radius~$r$ centered on the identity element in~$\cay(\Gamma,S)$.
This function depends on the choice of~$S$, but whether or not it is a polynomial,
or exponential function, is independent of~$S$.
This defines the classes of groups with polynomial or exponential growth.

By a celebrated result of Gromov, groups of polynomial growth are virtually nilpotent.
\begin{theorem}[Gromov \cite{gromov81polynomial}]
  \label{thm:gromov-polynomial-growth}
  If~$G$ is a group of polynomial growth, then~$G$ admits a nilpotent subgroup~$H$ of finite index.
\end{theorem}

\begin{proposition}
  \label{prop:polynomial-growth-tww}
  Groups of polynomial growth have finite uniform twin-width.
\end{proposition}
\begin{proof}
  Immediate by \cref{thm:gromov-polynomial-growth,prop:solvable-tww,cor:finite-index}.
\end{proof}

When~$H \lhd G$, \cref{lem:extension-gen} gives an upper bound on~$\utww(G)$ by~$\utww(H)$ and~$\utww(G/H)$.
There is also a trivial lower bound~$\utww(G) \ge \utww(H)$.
However~$\utww(G/H)$ gives no lower bound on~$\utww(G)$ in general:
it can be that~$\utww(G) = 2$ and~$G/H$ has infinite twin-width,
namely when~$G$ is a free group and~$G/H$, a finitely generated group with infinite twin-width.
We can however give a result when~$H$ is finite.

\begin{lemma}
  \label{lem:finite-quotient}
  Let~$G$ be a group and~$N \lhd G$ a normal subgroup with $\card{N} = t < \infty$.
  Then~$G$ has finite twin-width if and only if~$G/N$ does.
  Furthermore
  \begin{align*}
    \utww(G) & \le \max(\utww(G/N),t),~\text{and}\\
    \utww(G/N) & \le 2^{t^{O(\utww(G) \cdot t)}}.
  \end{align*}
\end{lemma}
\begin{proof}
  Using \cref{cor:extension} and the trivial bound~$\utww(N) \le t$,
  if~$G/N$ has finite twin-width then so does~$G$, and $\utww(G) \le \max(\utww(G/N),t)$.

  For the converse, consider a total order~$<$ on~$G$ witnessing its twin-width.
  For~$a \in G/N$, choose a representative~$\iota(a) \in a$.
  These are used to order~$G/N$ according to~$<$, by
  \[ \forall a,b \in G/N, \quad a <_{G/N} b \iff \iota(a) < \iota(b). \]
  For~$x \in G$, let~$M^G_x$ be the matrix of~$x$ acting on~$(G,<)$ by right product, and similarly in~$G/N$.
  Fix some~$a \in G/N$. We want a bound on~$\stww(M^{G/N}_a)$.

  Define the matrix~$S$ with~$(G/N,<)$ as row and column indices by
  \[ \forall b_1,b_2 \in G/N, \quad S(b_1,b_2) = 1 \iff \iota(b_1)^{-1} \iota(b_2) \in a. \]
  It can be seen as the superposition of~$M^G_x$ for~$x \in a$,
  restricted to the submatrix with~$\iota(G/N)$ as indices.
  If~$b_1 a = b_2$, then~$\iota(b_1)^{-1} \iota(b_2) \in a$,
  which in matrices means that~$M^{G/N}_a(b_1,b_2) = 1$ implies~$S(b_1,b_2) = 1$.
  It follows that~$\stww(M^{G/N}_a) \le \stww(S)$.

  Suppose now that~$\utww(G) \le k$ is witnessed by the order~$<$, meaning that for~$x \in G$, $\stww(M^G_x) \le k$.
  Using \cref{thm:grid-theorem-matrix,lem:grid-ramsey},
  and that~$S$ is contained in the superposition of~$t$ matrices~$M^G_x$,
  we obtain
  \[ \stww(M^{G/N}_a) \le \stww(S) \le 2^{t^{O(tk)}}. \]
  If~$G$ merely has finite twin-width witnessed by~$<$,
  there still is a bound on~$\stww(M^G_x)$ for~$x \in a$ with~$a$ fixed, because~$a$ is finite.
  The same arguments show that $\stww(M^{G/N}_a)$ is finite, hence~$G/N$ has finite twin-width.
\end{proof}

\subsection{Infinite products and wreath products}
Given~$(G_i)_{i \in I}$ a family of groups, their \emph{direct product} is denoted by~$\prod_{i \in I} G_i$.
Their \emph{direct sum}~$\bigoplus_{i \in I} G_i$ is the subgroup of the direct product
consisting of finitely supported tuples, i.e.\ the~$(x_i)_{i \in I}$ with~$x_i \in G_i$
such that all but finitely many~$x_i$ are the neutral element.
When~$I$ is finite, the two notions coincide.

\Cref{lem:extension-gen} of course implies that finite products preserve twin-width and uniform twin-width.
This can be generalised to infinite direct sum, and for uniform twin-width only, to infinite direct products.

\begin{lemma}
  \label{lem:direct-product}
  Given~$(G_i)_{i \in I}$ a possibly infinite family of groups,
  consider $S = \bigoplus_{i \in I} G_i$ their direct sum, and $P = \prod_{i \in I} G_i$ their direct product.
  Then~$S$ has finite twin-width if and only if all~$G_i$ do, and
  \[ \utww(S) = \utww(P) = \sup_{i \in I} \utww(G_i). \]
\end{lemma}
\begin{proof}
  Choose some well-order on~$I$, a total order~$<_i$ on each~$G_i$ witnessing its twin-width,
  and order~$S$ and~$P$ lexicographically.
  Assuming that~$\utww(G_i) \le k$ is witnessed by~$<_i$,
  \cref{lem:tww-tensor-product} immediately gives~$\utww(P) \le k$, and a fortiori~$\utww(S) \le k$.
  This proves the equalities on uniform twin-width.

  For the claim on non-uniform twin-width, assume that all~$G_i$ have finite twin-width witnessed by~$<_i$,
  and let~$x = (x_i)_{i \in I} \in S$.
  Again the matrix of~$x$ acting on~$(S,<)$ is the tensor product of the matrices of~$x_i$ acting on~$(G_i,<_i)$.
  Since~$x$ is finitely supported, all but finitely many of these matrices are identity,
  while the remaining ones have finite twin-width.
  Thus there is a universal bound on the twin-width of these matrices,
  and we conclude using \cref{lem:tww-tensor-product} once again.
\end{proof}

Given any groups~$G,H$, one can define the (complete) \emph{wreath product}
\[ G \wr H \eqdef G^H \rtimes H \]
where~$H$ acts by permuting indices by right product,
i.e.\ for~$h,h' \in H$ and~$x \in G^H$, the action of~$h$ is
\[ (x \cdot h)(h') \eqdef x(h'h). \]
One may also consider the restricted wreath product, with a direct sum replacing the direct product,
and the same action of~$H$.
\[ G \wr_r H \eqdef \left( \bigoplus_{h \in H} G \right) \rtimes H. \]
It is a subgroup of the complete wreath product.

\begin{lemma}
  \label{lem:wreath-product}
  $\utww(G \wr H) = \max(\utww(G), \utww(H)).$
\end{lemma}
\begin{proof}
  By \cref{lem:direct-product,cor:semidirect-product}.
\end{proof}

Interesting groups can be constructed through wreath products,
the most common example being the lamplighter group~$(\Zz / 2\Zz) \wr_r \Zz$,
which therefore has uniform twin-width 2.

\subsection{Group actions}
\label{sec:group-action}
The twin-width of a group is defined through its self-action by product.
In this section, we consider the twin-width of other actions,
and show a relation to the twin-width of the group.

\begin{lemma}
  \label{lem:well-ordered-permutations}
  Let~$(X,<_X)$ be a well-ordered set.
  Then there exists a total order~$<_{\Perm_X}$ on the permutation group~$\Perm_X$
  such that for~$\sigma \in \Perm_X$
  \[ \stww_{\Perm_X}(\sigma) \le \stww_X(\sigma) \]
  where the right-hand side is the twin-width of~$\sigma$ acting on~$(X,<_X)$,
  whereas the left-hand side is the twin-width of~$\sigma$ acting on~$(\Perm_X,<_{\Perm_X})$ by product.
\end{lemma}
\begin{proof}
  To respect the convention that twin-width is defined with regards to right product,
  we will use postfix notations for~$\Perm_X$, i.e.\ for~$\sigma,\tau \in \Perm_X$ and~$x \in X$,
  we write~$x \cdot \sigma = \sigma(x)$ and~$\sigma \tau = \tau \circ \sigma$.

  We order~$X^X$ lexicographically, i.e.\ for any~$f \neq g : X \to X$,
  define $m_{f,g} = \min \setst{x \in X}{x \cdot f \neq x \cdot g}$
  and
  \[ f <_{X^X} g \quad \iff \quad m_{f,g} \cdot f <_X m_{f,g} \cdot g. \]
  The group~$\Perm_X$ acts on~$X^X$ by post-composition,
  that is~$(f \cdot \sigma)(x) = f(x) \cdot \sigma$.
  We define~$<_{\Perm_X}$ by restricting the former order from~$X^X$ to~$\Perm_X$,
  while the restriction of this action is simply right product.

  For~$\sigma \in \Perm_X$, let~$\bij\sigma$ be the matrix of~$\sigma$ acting on~$(X,<_X)$,
  let~$M^X_\sigma$ be the matrix of~$\sigma$ acting on~$(X^X,<_{X^X})$ by post-composition,
  and let~$M^{\Perm_X}_\sigma$ be the matrix of~$\sigma$ acting on~$(\Perm_X,<_{\Perm_X})$ by right product.
  Then~$M^X_\sigma$ is a tensor product of~$\bij\sigma$ with itself,
  hence by \cref{lem:tww-tensor-product}, we have~$\stww(M^X_\sigma) = \stww(\bij\sigma)$.
  Furthermore~$M^{\Perm_X}_\sigma$ is clearly a submatrix of~$M^X_\sigma$, thus
  \[ \stww_{\Perm_X}(\sigma) = \stww(M^{\Perm_X}_\sigma) \le \stww(M^X_\sigma) = \stww(\bij\sigma) = \stww_X(\sigma). \]
\end{proof}

\begin{corollary}
  \label{cor:well-ordered-action}
  Let~$(X,<_X)$ be a well-ordered set, and~$G$ a group acting faithfully on~$X$ on the right.
  If this action has finite twin-width with regards to~$<_X$, then~$G$ has finite twin-width.
  Furthermore if this action has uniform twin-width~$k$ with regards to~$<_X$, then~$\utww(G) \le k$.
\end{corollary}
\begin{proof}
  Through the faithful action, $G$ is a subgroup of~$\Perm_X$.
\end{proof}

Let~$T$ be the infinite rooted binary tree.
Let~$\Aut(T)$ be the group of automorphisms of~$T$,
i.e.\ bijections on nodes of~$T$ which preserve the root and the parent relation.
The group~$\Aut(T)$ is uncountable, but it has some interesting finitely generated subgroups,
notably the Grigorchuk group~\cite{grigorchuk1980burnside}.

\begin{proposition}
  \label{prop:tww-tree-automorphism}
  The group~$\Aut(T)$ has uniform twin-width~2.
\end{proposition}
\begin{proof}
  Let~$<$ be the breadth-first search order on~$T$,
  which orders first by increasing depth, then left-to-right at each level.
  This is a well-order, hence by \cref{cor:well-ordered-action}
  it suffices to prove that the action of~$\Aut(T)$ on~$T$ with regards to~$<$ has uniform twin-width~2.

  Let~$f \in \Aut(T)$, and for~$d \in \Nn$ let~$f_d$ be the action of~$f$ restricted to nodes of depth~$d$.
  Then the bijection matrix~$\bij{f_d}$ is obtained by substituting
  the 2-by-2 diagonal and anti-diagonal matrices inside~$\bij{f_{d-1}}$.
  It follows by \cref{lem:tww-substitution} that~$\stww(\bij{f_d}) = 2$ for all~$d$.
  (These are \emph{separable} permutation matrices, see~\cite[Section~3.2]{Guillemot14}.)
  Finally, the matrix of~$f$ can be seen as the~$\Nn \times \Nn$ diagonal matrix,
  with coefficient~$(d,d)$ substituted by~$\bij{f_d}$, hence it has twin-width 2 by \cref{lem:tww-substitution} again.
\end{proof}

\subsection{Uniform and non-uniform twin-width}
Any group with finite uniform twin-width also has finite twin-width.
We conjecture that the converse does not hold,
and we think the following is a good candidate for a counterexample.

Let~$\Perm_\Zz$ be the group of permutations on~$\Zz$,
and let~$\Perm_\Zz^f$ be the subgroup consisting of finitely supported permutations.
Furthermore, let $T = \setst{x \mapsto x + c}{c \in \Zz}$ the subgroup of translations.
While~$\Perm_\Zz^f$ is not finitely generated, $\Perm_\Zz^f \rtimes T$ is.
The latter is sometimes called lampshuffler group.

Consider the well-order~$0,1,-1,2,-2,\dots$ on~$\Zz$.
The transposition of~$0$ and~$1$  and the translation by~$1$
are two permutations of~$\Zz$ which generate~$\Perm_\Zz^f \rtimes T$,
and whose action on~$\Zz$ clearly has finite twin-width.
Thus, by \cref{lem:perm-composition}, the action of the lampshuffler on~$\Zz$ has finite twin-width,
and by \cref{cor:well-ordered-action}, so does the lampshuffler group itself.

We conjecture that the lampshuffler has infinite uniform twin-width,
and thus separates uniform and non-uniform twin-width.
A justification for this conjecture is that the lampshuffler having finite uniform twin-width
would imply a universal bound on the uniform twin-width of finite groups,
and more generally of elementary amenable groups
(the groups obtained from finite groups and abelian groups by taking
subgroups, quotients, extensions, and direct limits).
This would be surprising.

\begin{proposition}
  \label{lem:utww-finite-groups}
  The following are equivalent for any~$k \in \Nn$.
  \begin{enumerate}
    \item \label{item:finite-groups} For all finite group~$G$, $\utww(G) \le k$.
    \item \label{item:lampshuffler} $\utww(\Perm_\Zz^f \rtimes T) \le k$.
    \item \label{item:elementary-amenable} For all elementary amenable group~$G$, $\utww(G) \le k$.
  \end{enumerate}
\end{proposition}
\begin{proof}
  The lampshuffler~$\Perm_\Zz^f \rtimes T$ is elementary amenable and contains all finite groups,
  hence~\eqref{item:elementary-amenable} implies~\eqref{item:lampshuffler} which implies~\eqref{item:finite-groups}.
  Suppose now that all finite groups have uniform twin-width at most~$k$,
  and recall that abelian groups have uniform twin-width~2.
  Remark that $k \ge 2$, because only the trivial group has twin-width~1.
  It is known that all elementary amenable groups can be obtained from finite groups and abelian groups
  using only extensions and direct limits~\cite{chou1980elementary}.
  It follows by \cref{cor:extension,lem:group-limit} that all elementary amenable groups have twin-width at most~$k$.
\end{proof}

One could replace $\Perm_\Zz^f \rtimes T$ with just $\Perm_\Zz^f$ in the result:
they have the same uniform twin-width by \cref{cor:semidirect-product}.

Remark that the previous result means that if all elementary amenable groups have finite uniform twin-width,
then there is in fact a universal bound on their uniform twin-width,
because the lampshuffler is elementary amenable.

\section{Constructing groups with infinite twin-width}
\label{sec:infinite-tww}
This section constructs a finitely generated group with infinite twin-width.

\subsection{Embedding graphs of high twin-width}
\label{sec:cubic-embedding}
The construction embeds an appropriate sequence of graphs with bounded degree and unbounded twin-width in a Cayley graph.
The existence of such graphs is given by counting arguments (see \cref{thm:tww-small}).
The embedding into a group uses the following result of Osajda, based on graphical small cancellation theory.

\begin{theorem}[{\cite[Theorem~3.2]{osajda2020small}}]
  \label{thm:graph-sequence-embedding}
  Consider constants~$A > 0$ and~$D \in \Nn$, and let~$(G_n)_{n \in \Nn}$ be a sequence of connected graphs such that
  \begin{enumerate}
    \item $\Delta(G_n) \le D$,
    \item $\girth(G_n) \ge \girth(G_{n-1}) + 6$, and
    \item $\diam(G_n)/\girth(G_n) \le A$.
  \end{enumerate}
  Then there exists a group~$\Gamma$ finitely generated by~$S$ such that
  every~$G_n$ isometrically embeds in~$\cay(\Gamma,S)$.
\end{theorem}

To obtain a finitely generated group of infinite twin-width,
it suffices to construct a sequence of graphs satisfying these hypotheses,
and which furthermore has unbounded twin-width.
Precisely, the sequence we construct satisfies the following\footnotemark{}:
\begin{enumerate}
  \item Each~$G_n$ has maximum degree at most~6,
  \item $\diam(G_n) \le 3 \log \card{V(G_n)}$, and
  \item $\girth(G_n) > \frac{1}{4} \log \card{V(G_n)}$.
\end{enumerate}
\footnotetext{Logarithms in this section are in base~2.}

For~$n \in \Nn$, we write~$[n] = \{1,\dots,n\}$.
For~$n$ even, let~$\Cc_1(n)$ be the class of edge-labeled graphs on vertex set~$[n]$
which are formed by the union of three perfect matchings with labels~1, 2, and~3
(these graphs may have multiple edges).
This class is not small.
\begin{lemma}
  \label{lem:cubic-not-small}
  For any constant~$c$ and any sufficiently large even~$n$,
  \[ \card{\Cc_1(n)} > n! \cdot c^n. \]
\end{lemma}
\begin{proof}
  The class~$\Cc_1(n)$ contains all possible choices of 3 perfect matchings
  from~$\{1,\dots,n/2\}$ to~$\{n/2+1,\dots,n\}$.
  There are~$(n/2)! ^3$ such choices, i.e.\ at least~$n!^{3/2}$ up to some exponential factors,
  which is asymptotically larger than~$n! \cdot c^n$.
\end{proof}

Fix~$n \in \Nn$ even and define $g \eqdef \log(n) / 4$.
Define~$\Cc_2(n)$ the class of graphs on vertex set~$[n]$ with degree at most~6,
diameter at most~$3\log n$ and girth at least~$g$.
Let us show that at least half of~$\Cc_1(n)$ can be obtained from graphs of~$\Cc_2(n)$
by editing (adding or removing) at most~$n^{7/8}$ edges and adding labels~$1,2,3$ to edges.

\begin{lemma}
  \label{lem:expectation-cycles}
  If~$G$ is a graph chosen uniformly at random in~$\Cc_1(n)$,
  then the expected number of cycles in~$G$ of length at most~$g$ is at most~$2 \cdot 6^g$.
\end{lemma}
\begin{proof}
  Let~$C$ be a potential cycle of length~$\ell \le g$ properly edge-colored with~$1,2,3$;
  that is, a sequence~$v_1,\dots,v_\ell$ of vertices in~$[n]$ where each
  potential edge~$v_iv_{i+1}$ (with indices modulo~$\ell$) is given a label in~$1,2,3$,
  and such that consecutive edges have distinct labels.
  Let us bound the probability that~$C$ appears as a cycle with the assigned colors in~$G$.

  When choosing a random perfect matching~$M$ on~$n$ vertices ($n$ even),
  the probability that a given pair of vertices belongs to~$M$
  is $\frac{n}{2} / \binom{n}{2} = \frac{1}{n-1}$.
  If these probabilities were independent, $C$ would appear in~$G$ with probability exactly~$(n-1)^{-\ell}$.
  But the fact that each color class of~$G$ forms a matching introduces a small dependency:
  Conditioned by the fact that~$t$ edges of a random perfect matching~$M$ are already known,
  the probability that a~given pair of not yet matched vertices appears in~$M$ is~$1/(n-2t-1)$.
  Now if these already known edges are part of~$C$, then~$t \le \ell \le g \ll n$,
  hence this probability is upper bounded by~$2/n$.
  Thus the probability that~$C$ appears in~$G$ is at most~$(2/n)^{\ell}$.
  Therefore the expected total number of cycles of length at most~$g$ in~$G$ is at most
  \[ \sum_{i=2}^g 3^i \cdot n^i \cdot (2/n)^{i} \le 2 \cdot 6^g. \]
\end{proof}

\begin{lemma}
  \label{lem:editing}
  For any even~$n$, for at least half of the graphs~$G=(V,E)$ in~$\Cc_1(n)$,
  there exists~$G'=(V,E')$ in~$\Cc_2(n)$ such that $|E \symdiff E'|=O(n^{7/8})$,
  where~$\symdiff$ denotes the symmetric difference.
\end{lemma}
\begin{proof}
  By \cref{lem:expectation-cycles} and Markov's inequality,
  at least half of the graphs~$G_1$ in~$\Cc_1(n)$ have at most~$4 \cdot 6^g$ cycles of length at most~$g$.
  Let~$G_1$ be such a~graph, and choose~$F \subset E(G_1)$ of size at most~$4 \cdot 6^g$
  intersecting every cycle of length at most~$g$ in~$G_1$ (simply choose an edge in each such cycle).
  Let~$G_2$ be the subgraph of~$G_1$ obtained by deleting the edges of~$F$.
  By construction, $G_2$ has girth more than~$g$.
  We will now add edges to~$G_2$ to ensure its diameter is at most~$3 \log n$
  without creating new short cycles.

  Let~$Z$ be the set of endpoints of edges in~$F$. Thus~$\card{Z} \le 8 \cdot 6^g$.
  Let~$N$ be the union of closed balls of radius~$g/2$ centered on vertices of~$Z$.
  Since~$G_2$ has maximum degree 3, for any fixed vertex~$x$,
  the ball of radius~$d$ around~$x$ in~$G_2$ contains at most~$3 \cdot 2^d$ vertices.
  It follows that
  \begin{equation}
    \card{N} \le 8 \cdot 6^g \cdot 3 \cdot 2^{g/2} = 24 \cdot (6 \sqrt{2})^g \le 24 \cdot n^{\log(6 \sqrt{2}) / 4}.
  \end{equation}
  Observe that $\log(6 \sqrt{2}) / 4 \le 7/8$, hence
  \begin{equation}
    \card{N} \le 24 \cdot n^{7/8}.
    \label{eq:size-N}
  \end{equation}
  Choose~$X$ an inclusion-wise maximal subset of~$V(G_2)$ such that
  vertices in~$X$ pairwise have distance at least~$\log n$ in~$G_2$.
  Define~$X_1 = X \cap N$, and~$X_2 = X \setminus N$.

  Let~$v \in X_2$, and let~$B_v$ be the closed ball of radius~$g/2$ in~$G_2$.
  Since~$v$ is at distance more than~$g/2$ from~$Z$,
  no edge of~$G_1$ was removed inside this ball, hence all vertices of~$B_v$ have degree~3.
  Furthermore, because~$G_2$ has girth more than~$g$, $G_2$ restricted to~$B_v$ is a tree,
  except possibly for edges between vertices at distance $g/2$ from $v$ which we may ignore.
  This implies
  \begin{equation}
    \card{B_v} \ge 2^{g/2} = n^{1/8}.
    \label{eq:size-balls-X2}
  \end{equation}
  Finally, any distinct~$v,v'\in X_2$ are at distance at least $\log n$, which is more than $g$.
  Thus the balls $B_v$ and $B_{v'}$ are disjoint.
  It follows from~\cref{eq:size-balls-X2} that
  \begin{equation}
    \card{X_2} \le \frac{\card{V(G_2)}}{n^{1/8}} = n^{7/8}.
    \label{eq:size-X2}
  \end{equation}
  Using that $X = X_1 \cup X_2$, $X_1 \subset N$, and \cref{eq:size-N,eq:size-X2}, we obtain
  \begin{equation}
    \card{X} = O(n^{7/8}).
  \end{equation}

  Finally, let~$G_3$ be a graph obtained by adding to~$G_2$ a balanced ternary tree~$T$ with vertex set~$X$.
  Thus, vertices in~$X$ are pairwise at distance at most~$\log n$ in~$G_3$.
  Furthermore, by maximality of~$X$, all vertices of~$G_2$ are at distance at most~$\log n$ of~$X$.
  It follows that the diameter of~$G_3$ is at most~$3 \log n$.
  Finally, $G_3$ still has girth at least~$g$, because~$T$ only connects vertices of~$X$, which are far apart in~$G_2$.
  Clearly the degree of vertices in~$G_3$ does not exceed 6.
  Thus~$G_3 \in \Cc_2(n)$.

  To summarize, we transform~$G_1$ into~$G_3$ by
  deleting $\card{F} \le 4 \cdot 6^g$ edges, and then adding $\card{X}$ edges,
  which gives~$O(n^{7/8})$ edge editions in total.
\end{proof}

\begin{lemma}
  \label{lem:cubic-girth-diam-not-small}
  The class $\Cc_2$ is not small. Precisely, for any constant~$c$ and any sufficiently large even~$n$,
  \[ \card{\Cc_2(n)} \ge n! \cdot c^n. \]
\end{lemma}
\begin{proof}
  Fix~$c$, consider~$n$ even sufficiently large,
  and suppose for a contradiction that $\card{\Cc_2(n)} \le n! \cdot c^n$.
  For any graph~$G$ on~$n$ vertices,
  the number of graphs which can be obtained from~$G$ by editing up to~$n^{7/8}$ edges
  is at most
  \begin{equation}
    \binom{n^2}{n^{7/8}} \le n^{2n^{7/8}} \le 2^n,~\text{asymptotically}.
  \end{equation}
  Hence the total number of cubic graphs and edges labeled~1, 2, 3
  obtained by~$n^{7/8}$ editions from any graph in~$\Cc_2(n)$ is asymptotically at most
  \begin{equation}
    2^n \cdot 3^{3n/2} \cdot n! \cdot c^n \le n! \cdot c_2^n
  \end{equation}
  for some constant~$c_2$.
  By \cref{lem:editing}, the cardinality of~$\Cc_1(n)$ is at most double the previous value,
  a contradiction of \cref{lem:cubic-not-small}.
\end{proof}

\infinitetww*

\begin{proof}
  By \cref{lem:cubic-girth-diam-not-small,thm:tww-small},
  for any~$k \in \Nn$ and for any sufficiently large~$n_k$, $\Cc_2(n_k)$ contains graphs of twin-width more than $k$.
  Thus, one can construct a sequence~$(G_k)_{k \in \Nn}$ of graphs such that~$\tww(G_k) > k$ and~$G_k \in \Cc_2(n_k)$,
  for any sequence~$(n_k)_{k \in \Nn}$ of even integers growing sufficiently fast.
  Since $\girth(G_k) \ge \log(n_k)/ 4$, an appropriate choice of $n_k$
  ensures that $\girth(G_k) \ge \girth(G_{k-1}) + 6$.
  Then~$(G_k)_{k \in \Nn}$ satisfies the hypotheses of \cref{thm:graph-sequence-embedding},
  hence there exists a group~$\Gamma$ finitely generated by~$S$
  such that every~$G_k$ isometrically embeds in~$\cay(\Gamma,S)$.
  This group has infinite twin-width.
\end{proof}

\subsection{Computability considerations}
In this section, we argue that the proof of \cref{thm:infinite-tww} is effective,
and in particular that it can be turned into an algorithm to decide the word problem in the resulting group.

\begin{theorem}
  \label{thm:infinite-tww-decidable}
  There exists a group with infinite twin-width given by a presentation~$\grppres{S}{R}$,
  with a~finite set~$S$ of generators and a~recursive set $R$ of relators.
  Furthermore, this group has a decidable word problem.
\end{theorem}

Using Higman embedding theorem~\cite{higman1961subgroups}, we obtain the following.
\begin{corollary}
  There is a finitely presented group with infinite twin-width.
  \label{them:infinite-tww-finite-pres}
\end{corollary}
Note that this finitely presented group is a supergroup of the one constructed in \cref{thm:infinite-tww-decidable},
and might not have a decidable word problem.

To prove \cref{thm:infinite-tww-decidable},
we split the proof of \cref{thm:infinite-tww} in three steps,
and prove that each of them is effective.
\begin{enumerate}
  \item Constructing the sequence of graphs~$(G_k)_{k \in \Nn}$,
    with appropriate conditions on degree, girth, diameter, and twin-width (\cref{sec:cubic-embedding}).
  \item Adding labels to this sequence satisfying the small cancellation conditions.
    This is the core of Osajda's proof, see~\cite[Section~2]{osajda2020small}.
  \item Obtaining a group from this small cancellation graph labelling.
    This construction, attributed to Gromov~\cite{gromov2003random},
    is based on the well-known techniques of small cancellation theory
    (see e.g.~\cite[Chapter~V]{lyndon1977group}).
\end{enumerate}

\subsubsection{Computing the sequence of graphs}
\begin{lemma}
  \label{thm:tww-girth-diam-effective}
  There exists a computable sequence of graphs~$(G_k)_{k \in \Nn}$,
  satisfying the hypotheses of \cref{thm:graph-sequence-embedding}, and with unbounded twin-width.
\end{lemma}
\begin{proof}
  By \cref{lem:cubic-girth-diam-not-small,thm:tww-small},
  one can find such a sequence by choosing a~$G_k \in \Cc_2(n_k)$ appropriately, for some $n_k$ function of~$k$.
  It is easy to obtain an explicit formula for~$n_k$ from
  the proofs of \cref{lem:expectation-cycles,lem:editing,lem:cubic-girth-diam-not-small}.
  This yields a naive algorithm: Enumerate all graphs of size~$n_k$
  until finding one in~$\Cc_2(n_k)$ with sufficiently large twin-width.

  However, the proofs give a natural and much more efficient randomized algorithm.
  One chooses a graph uniformly at random in~$\Cc_1(n_k)$.
  Edge editions are performed according to \cref{lem:editing},
  and the process is aborted should it require more than~$4 \cdot 6^g$ edge deletions---%
  which by \cref{lem:expectation-cycles} happens with probability less than~$1/2$.
  In case of success, the resulting graph~$G_k$ is in~$\Cc_2(n_k)$.
  Although the distribution of~$G_k$ over~$\Cc_2(n_k)$ is not uniform,
  the proof of \cref{lem:cubic-girth-diam-not-small} can be adapted to show that
  for any constant~$c$ and any subset $X \subset \Cc_2(n_k)$ of size $\card{X} \le n! \cdot c^n$,
  the probability that~$G_k$ is in~$X$ tends to~0 as $n_k$ grows.
  Applying this to the subset~$X$ of graphs with twin-width at most~$t$ for any fixed~$t$,
  we obtain by \cref{thm:tww-small} that the sequence~$(G_k)_{k \in \Nn}$ has unbounded twin-width with probability~1.
  Thus, this probabilistic algorithm does not even need to compute the twin-width of any graph.

  This algorithm chooses~$G_k$ in expected time polynomial in its size~$n_k$.
  However~$n_k$ needs to be at least exponential in~$k$,
  since we want~$G_k$ to have girth at least linear in~$k$, but at most logarithmic in its size.
\end{proof}

\subsubsection{Computing the small cancellation labelling}
Osajda (\cite[Section~2]{osajda2020small}) equips the sequence~$(G_k)_{k \in \Nn}$
with an edge labelling satisfying the small cancellation condition:
if~$p,p'$ are two distinct paths in~$G_k,G_{k'}$ respectively,
with the same length and edge labelling,
then $|p| < \lambda \cdot \girth(G_k)$, and $|p'| < \lambda \cdot \girth(G_{k'})$,
where~$\lambda > 0$ is some fixed constant (the small cancellation parameter).
Said shortly: Any two paths with the same labelling must be short, relative to the girth.

The proof makes large use of Lovász Local Lemma.
Assuming some appropriate labellings of~$G_0,\dots,G_{k-1}$ have been chosen inductively,
one proves that a random labelling of~$G_k$ is `good' with non-zero probability.
Crucially, this is an incremental process:
to choose the labelling of~$G_k$, one only needs to consider~$G_0,\dots,G_k$.
It will always be possible to later extend this sequence with an appropriate labelling of~$G_{k+1}$.
This readily gives a~naive algorithm, which enumerates all labellings of~$G_k$
and chooses the first `good' one, before continuing with~$G_{k+1}$.
Effective algorithms for Lovász Local Lemma such as~\cite{moser2010constructive}
may be helpful to improve the complexity.

\subsubsection{Deciding the word problem}
The former labelling of~$(G_k)_{k \in \Nn}$ uses labels in~$S \uplus S^{-1}$,
where~$S$ is some finite set, and~$S^{-1}$ denotes the set of formal inverses of~$S$.
The group constructed by \cref{thm:graph-sequence-embedding} is then
\[ \Gamma = \grppres{S}{R} \]
where the set~$R$ of relators contains all words on~$S \uplus S^{-1}$ read along some cycle of some~$G_k$.
This set~$R$ is recursive: given some word~$w$ over~$S \uplus S^{-1}$,
to decide if~$w \in R$, one must test if~$w$ appears as the labelling of some cycle of some~$G_k$.
This can only happen if~$\girth(G_k) \le |w|$.
Since~$\girth(G_k)$ increases at least linearly with~$k$,
it suffices to check a finite number of graphs in the sequence.
This proves the first half of \cref{thm:infinite-tww-decidable}.

Decidability of the word problem uses a classical result of small cancellation theory,
stated below for the case of finite graphical presentations.
\begin{lemma}[Greendlinger's lemma~\cite{ollivier2006small}]
  \label{lem:greendlinger}
  Let~$G$ be a finite graph equipped with an edge labelling on~$S \uplus S^{-1}$
  satisfying the small cancellation condition.
  Let~$R$ be the set of words read along the cycles of~$G$, and~$\Gamma = \grppres{S}{R}$.
  If~$w$ is a word over~$S \uplus S^{-1}$ which represents the identity of~$\Gamma$,
  then some relator~$r \in R$ can be applied to~$w$ to reduce it to~$w'$, with~$\card{w'} < \card{w}$.
\end{lemma}

Although the above is stated for a finite graph, it generalises to a sequence of finite graphs.
Suppose a word~$w$ over~$S \uplus S^{-1}$ represents the identity of~$\Gamma$.
Then there is finite sequence of relators in~$R$ which can be applied to~$w$ to reduce it to the empty word.
These relators are found in finitely many graphs in the sequence~$(G_k)_{k \in \Nn}$, say~$G_0,\dots,G_s$.
Then~$w$ is also equal to~$1$ in the group~$\grppres{S}{R'}$ where the relators~$R'$
are the ones read along some cycle of~$G_0 \cup \dots \cup G_s$.
We conclude by applying \cref{lem:greendlinger} to~$\grppres{S}{R'}$.

Where Greendlinger's lemma applies, the word problem can be decided using Dehn's algorithm:
given a word~$w$ on~$S \cup S^{-1}$, try to rewrite it into a strictly smaller word
by applying some relator of~$R$, and repeat.
If this fails before reaching the empty word, Greendlinger's lemma proves
that the word is not equivalent to the neutral element in~$\Gamma$.
This algorithm works for~$\Gamma$ despite the set of relators~$R$ being infinite.
Indeed, given~$w$ of length~$k$, only relators of length less than~$2k$ need to be considered in Dehn's algorithm.
There are only finitely many such relators, and they can be enumerated.
This concludes the proof of \cref{thm:infinite-tww-decidable}

\section{Remarks on queue number}
\label{sec:queue-number}
This final section draws some remarks on the relation between twin-width and queue number in the context of groups.
Many results of this paper hold when replacing twin-width with queue number.
We point out the proofs which fail for queue number, but without counterexamples:
separating bounded twin-width and bounded queue number on sparse graphs remains an open question.

\subsection{Definition}
Consider a graph~$G$ with a total order~$<$ on~$V(G)$.
Two edges~$xy$ and~$uv$ are said to be \emph{nested} if~$x < u < v < y$, up to symmetries.
The graph~$G$ is said to be a \emph{queue} with regards to~$<$ if it has no nested edges,
and a strict queue if furthermore it is a matching (no two edges share an endpoint).
A (strict) $t$\emph{-queue-layout} of~$G$ is a choice of order~$<$ on~$V(G)$
and a partition $E(G) = \biguplus_{i=1}^t E_i$ of its edges such that for each~$i$,
the graph formed by the edges~$E_i$ is a (strict) queue with regards to~$<$, see \cref{fig:queue-def}.
The (strict) \emph{queue number} of~$G$, denoted~$\qn(G)$ (resp.~$\sqn(G)$)
is the smallest~$t$ such that~$G$ admits a (strict) $t$-queue-layout.
Remark that the strict queue number is at least the maximum degree.
Furthermore, if~$G$ has maximum degree~$\Delta$, given a $t$-queue-layout,
refining the partition of edges with a proper edge coloring gives a strict $(\Delta+1)t$-queue layout
using Vizing's theorem. Thus
\begin{equation}
  \qn(G) \le \sqn(G) \le (\Delta+1)\qn(G),
\end{equation}
i.e.\ the two numbers are equivalent for graphs of bounded degree up to a constant factor.
We will focus on the strict queue number.

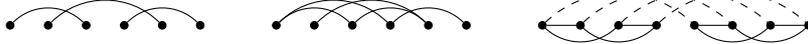
\begin{figure}
  \begin{center}
    \begin{tikzpicture}[scale=0.5,out=45,in=135, relative]
      \tikzstyle{every node}=[normalnode, minimum width=1mm]
      \foreach \i in {1,...,6}{
        \node (v\i) at (\i,0) {};
      }
      \draw (v1) to (v3);
      \draw (v2) to (v5);
      \draw (v4) to (v6);

      \begin{scope}[xshift=7cm]
      \foreach \i in {1,...,6}{
        \node (v\i) at (\i,0) {};
      }
      \draw (v1) to (v3);
      \draw (v3) to (v5);
      \draw (v1) to (v4);
      \draw (v4) to (v6);
      \draw (v2) to (v5);
      \end{scope}

      \begin{scope}[xshift=14cm]
      \foreach \i in {1,...,8}{
        \node (v\i) at (\i,0) {};
      }
      \foreach \i in {1,3,5,7}{
        \draw (v\i) -- +(1,0);
      }
      \foreach \i in {1,2,5,6}{
        \draw (v\i) to[out=-45,in=-135] +(2,0);
      }
      \foreach \i in {1,2,3,4}{
        \draw[dashed] (v\i) to +(4,0);
      }
      \end{scope}
    \end{tikzpicture}
  \end{center}
  \caption{%
    A strict queue; a queue; and a 2-queue-layout of the cube,
    partitioned in solid and dashed edges.
  }
  \label{fig:queue-def}
\end{figure}

\subsection{Matrix characterisation}
Say that an~$X \times Y$ matrix~$M$ is increasing if it is the matrix of an increasing partial map,
i.e.\ whenever there are 1-coefficients at positions~$(x_1,y_1)$ and~$(x_2,y_2)$,
either~$x_1<x_2$ and~$y_1<y_2$, or~$x_2<x_1$ and~$y_2 < y_1$.
In particular, each row and column contains at most a single~`1'.

Call strict queue number of~$M$ the minimum~$k$
such that~$M$ is equal to the superposition of~$k$ increasing matrices.
If~$G$ is a graph and~$<$ an order on~$V(G)$ witnessing the strict queue number of~$G$,
then one verifies
\begin{equation}
  \sqn(G) \le \sqn(\adj{<}{G}) \le 2\sqn(G).
\end{equation}
The first inequality is simple, and the second is obtained by using up to $\sqn(G)$ increasing matrices
to cover the 1-entries above the diagonal, and as many for the 1-entries below the diagonal.
Hence the strict queue number of~$G$ is within a factor~2 of
the strict queue number of its adjacency matrix, assuming an optimal choice of order on~$V(G)$.

\subsection{Queue number of groups}
If one considers the matrix product, with addition interpreted as logical `or' and multiplication as logical `and',
it is clear that a product of increasing matrices is itself increasing.
Furthermore, the adjacency matrix of~$G^{(k)}$ is equal to~$(M + I)^k - I$
where~$M$ is the adjacency matrix of~$G$ and~$I$ is the identity.
Matrix product distributes over matrix superposition, thus assuming that~$\sqn(M) \le t$, we have that~$(M+I)$ is superposition of~$t+1$ increasing matrices,
and after distributing~$(M+I)^k$ is superposition of at most~$(t+1)^k$ increasing matrices.
Thus we obtain a result similar to \cref{lem:power-graph}:
$\sqn(G^{(k)})$ is bounded by a function of~$\sqn(G)$ and~$k$.
With a few other arguments, one can reproduce the proof of \cref{lem:tww-regular-embedding}
to prove that regular embeddings preserve finite strict queue number, which is therefore a group invariant.
All examples of \cref{sec:geometric-tww} can easily be shown to have finite queue number
(finitely generated abelian groups, finite Cartesian products of groups with finite queue number, hyperbolic groups).

The characterisation of groups with finite queue number through matrices
and the generalisation to infinitely generated groups work without any issue.
The uniform queue number of a group~$G$ is defined as
\[ \uqn(G) \eqdef \min_{\text{$<$ order on $G$}} \ \sup_{x \in G} \ \sqn(M_x^<) \]
with~$M_x^<$ the matrix of~$x$ acting on~$(G,<)$.
Groups with uniform queue number~1 are exactly right-orderable groups.
As in \cref{lem:group-limit}, direct limits preserve uniform queue number.

\subsection{Group constructions}
Cartesian products---and more generally group extensions---is where the first difference with twin-width appears.
Twin-width on permutation matrices is remarkable in that it is exactly preserved by substitution,
in the sense that the twin-width of the resulting matrix is
the maximum of the twin-width of the input matrices (see~\cref{lem:tww-substitution}).
This fails for queue number: if~$M,N$ are permutation matrices
and~$M[N]$ is obtained by substituting~$N$ into every 1-coefficient of~$M$,
then
\begin{equation}
  \sqn(M[N]) = \sqn(M) \sqn(N).
  \label{eq:substitution-queue-number}
\end{equation}
Thus, while Cartesian products (or group extensions) still preserve having finite uniform queue number,
they do not a priori preserve the exact bound.
Precisely, the lexicographic ordering used in~\cref{lem:extension-gen}
gives the bound $\uqn(G \times H) \le \uqn(G) \uqn(H)$ by \eqref{eq:substitution-queue-number}, 
and this bound is tight for this specific choice of order.
However, there might be a better choice of order on~$G \times H$
giving $\uqn(G \times H) = \max(\uqn(G),\uqn(H))$ in some cases.

For example, all finitely generated abelian groups have finite uniform queue number,
but it is unclear whether there is a universal bound on the latter.
In turn, infinitely generated abelian groups (and thus solvable groups) might have infinite uniform queue number.

Summarizing, proofs of finite twin-width from this paper also give finite queue number in the following cases.
\begin{enumerate}
  \item Hyperbolic groups have finite queue number.
  \item Right-orderable groups have uniform queue number 1, and this is a characterisation.
  \item Finitely generated abelian groups have finite uniform queue number,
    but a priori without a universal bound.
  \item Polycyclic groups, which are defined like solvable groups except
    requiring that the iterated quotients be finitely generated and abelian,
    have finite uniform queue number.
  \item Groups of polynomial growth have finite uniform queue number,
     since they are virtually nilpotent, and nilpotent groups are polycyclic.
\end{enumerate}
On the other hand, we conjecture that arbitrary abelian groups, and a fortiori solvable groups,
can have infinite queue number (uniform and non-uniform).
As a simple example, we do not know whether an infinite countable product of $\Zz / 2\Zz$ with itself has finite queue number.

We also conjecture that the group of tree automorphisms
considered in \cref{prop:tww-tree-automorphism} has infinite queue number.
This group is related to the Bilu-Linial expanders~\cite{Bilu06},
which are known to have twin-width at most~6~\cite[Section~5]{twin-width2},
and are conjectured to have unbounded queue number.

Indeed, these expanders can be described as follows.
Let~$G_n$ be the group of permutations on~$2^n$ elements obtained by
restricting automorphisms of the infinite binary tree~$T$ to nodes at depth~$n$ in~$T$, ordered left to right.
Equivalently, permutations in~$G_n$ are the ones obtained from a permutation~$\sigma \in G_{n-1}$
by substituting each `1'-coefficient of~$\sigma$ by either the diagonal or antidiagonal 2-by-2 matrix,
the choice being done independently for all `1'-coefficients.
Then Bilu-Linial expanders have a partition of vertices into $V_1,V_2,V_3,V_4$ of size~$2^n$ each,
and a choice of orders~$<_i$ on~$V_i$ such that the adjacency matrix of $(V_i,<_i)$ versus $(V_j,<_j)$
is the matrix of a permutation in~$G_n$.
Precisely, a random choice of permutations in~$G_n$ in this description yields an expander with high probability,
while Bilu and Linial give a derandomized construction.

From this description, it is simple to see that if the action of~$\Aut(T)$ on~$T$ has finite uniform queue number,
then Bilu-Linial expanders have bounded queue number.
The converse, and the relation with the queue number of~$\Aut(T)$ are not clear,
but it is likely that an answer to either question would give significant insight for the other.

\bibliographystyle{plain}
\bibliography{biblio}

\begin{thebibliography}{10}

\bibitem{Bilu06}
Yonatan Bilu and Nathan Linial.
\newblock Lifts, discrepancy and nearly optimal spectral gap.
\newblock {\em Combinatorica}, 26(5):495--519, 2006.

\bibitem{twin-width8}
{\'{E}}douard Bonnet, Dibyayan Chakraborty, Eun~Jung Kim, Noleen K{\"{o}}hler,
  Raul Lopes, and St{\'{e}}phan Thomass{\'{e}}.
\newblock Twin-width {VIII:} delineation and win-wins.
\newblock {\em CoRR}, abs/2204.00722, 2022.

\bibitem{twin-width2}
{\'{E}}douard Bonnet, Colin Geniet, Eun~Jung Kim, St{\'{e}}phan Thomass{\'{e}},
  and R{\'{e}}mi Watrigant.
\newblock Twin-width {II:} small classes.
\newblock In {\em Proceedings of the 2021 ACM-SIAM Symposium on Discrete
  Algorithms (SODA 2021)}, pages 1977--1996, 2021.

\bibitem{twin-width3}
{\'{E}}douard Bonnet, Colin Geniet, Eun~Jung Kim, St{\'{e}}phan Thomass{\'{e}},
  and R{\'{e}}mi Watrigant.
\newblock Twin-width {III:} max independent set, min dominating set, and
  coloring.
\newblock In Nikhil Bansal, Emanuela Merelli, and James Worrell, editors, {\em
  48th International Colloquium on Automata, Languages, and Programming,
  {ICALP} 2021, July 12-16, 2021, Glasgow, Scotland (Virtual Conference)},
  volume 198 of {\em LIPIcs}, pages 35:1--35:20. Schloss Dagstuhl -
  Leibniz-Zentrum f{\"{u}}r Informatik, 2021.

\bibitem{twin-width4}
{\'{E}}douard Bonnet, Ugo Giocanti, Patrice~Ossona de~Mendez, Pierre Simon,
  St{\'{e}}phan Thomass{\'{e}}, and Szymon Toru\'{n}czyk.
\newblock Twin-width {IV:} ordered graphs and matrices.
\newblock {\em CoRR}, abs/2102.03117, 2021, to appear at STOC 2022.

\bibitem{twin-width1}
{\'{E}}douard Bonnet, Eun~Jung Kim, St{\'{e}}phan Thomass{\'{e}}, and
  R{\'{e}}mi Watrigant.
\newblock Twin-width {I:} tractable {FO} model checking.
\newblock {\em J. {ACM}}, 69(1):3:1--3:46, 2022.

\bibitem{buyalo2007embedding}
Sergei Buyalo, Alexander Dranishnikov, and Viktor Schroeder.
\newblock Embedding of hyperbolic groups into products of binary trees.
\newblock {\em Inventiones mathematicae}, 169(1):153--192, 2007.

\bibitem{chou1980elementary}
Ching Chou.
\newblock Elementary amenable groups.
\newblock {\em Illinois Journal of Mathematics}, 24(3):396--407, 1980.

\bibitem{Cibulka16}
Josef Cibulka and Jan Kyncl.
\newblock {F{\"{u}}redi-Hajnal limits are typically subexponential}.
\newblock {\em CoRR}, abs/1607.07491, 2016.

\bibitem{ding1995decomposition}
Guoli Ding and Bogdan Oporowski.
\newblock Some results on tree decomposition of graphs.
\newblock {\em Journal of Graph Theory}, 20(4):481--499, 1995.

\bibitem{dujmovic2005layout}
Vida Dujmovic, Pat Morin, and David~R Wood.
\newblock Layout of graphs with bounded tree-width.
\newblock {\em SIAM Journal on Computing}, 34(3):553--579, 2005.

\bibitem{eppstein2022threedimensional}
David Eppstein, Laura Merker, Sergey Norin, Micha{\l}~T Seweryn, David~R Wood,
  et~al.
\newblock Three-dimensional graph products with unbounded stack-number.
\newblock {\em arXiv preprint arXiv:2202.05327}, 2022.

\bibitem{Fox13}
Jacob Fox.
\newblock {Stanley-Wilf} limits are typically exponential.
\newblock {\em CoRR}, abs/1310.8378, 2013.

\bibitem{ganley2001pagenumber}
Joseph~L Ganley and Lenwood~S Heath.
\newblock {The pagenumber of k-trees is O(k)}.
\newblock {\em Discrete Applied Mathematics}, 109(3):215--221, 2001.

\bibitem{grigorchuk1980burnside}
Rostislav~Ivanovich Grigorchuk.
\newblock Burnside problem on periodic groups.
\newblock {\em Funktsional'nyi Analiz i ego Prilozheniya}, 14(1):53--54, 1980.

\bibitem{gromov81polynomial}
Mikhael Gromov.
\newblock Groups of polynomial growth and expanding maps (with an appendix by
  {J}acques {T}its).
\newblock {\em Publications Math\'ematiques de l'IH\'ES}, 53:53--78, 1981.

\bibitem{gromov1987hyperbolic}
Mikhael Gromov.
\newblock Hyperbolic groups.
\newblock In {\em Essays in group theory}, pages 75--263. Springer, 1987.

\bibitem{gromov2003random}
Mikhail Gromov.
\newblock Random walk in random groups.
\newblock {\em Geometric \& Functional Analysis GAFA}, 13(1):73--146, 2003.

\bibitem{Guillemot14}
Sylvain Guillemot and D{\'{a}}niel Marx.
\newblock Finding small patterns in permutations in linear time.
\newblock In {\em Proceedings of the Twenty-Fifth Annual {ACM-SIAM} Symposium
  on Discrete Algorithms, {SODA} 2014, Portland, Oregon, USA, January 5-7,
  2014}, pages 82--101, 2014.

\bibitem{higman1961subgroups}
Graham Higman.
\newblock Subgroups of finitely presented groups.
\newblock {\em Proceedings of the Royal Society of London. Series A.
  Mathematical and Physical Sciences}, 262(1311):455--475, 1961.

\bibitem{kuske2005logical}
Dietrich Kuske and Markus Lohrey.
\newblock Logical aspects of {C}ayley-graphs: the group case.
\newblock {\em Annals of Pure and Applied Logic}, 131(1-3):263--286, 2005.

\bibitem{lyndon1977group}
Roger~C Lyndon and Paul~E Schupp.
\newblock {\em Combinatorial group theory}, volume 188.
\newblock Springer, 1977.

\bibitem{malitz1994pagenumber}
Seth~M Malitz.
\newblock Graphs with {E} edges have pagenumber {$O(\sqrt{E})$}.
\newblock {\em Journal of Algorithms}, 17(1):71--84, 1994.

\bibitem{MarcusT04}
Adam Marcus and G{\'{a}}bor Tardos.
\newblock Excluded permutation matrices and the {Stanley-Wilf} conjecture.
\newblock {\em J. Comb. Theory, Ser. {A}}, 107(1):153--160, 2004.

\bibitem{moser2010constructive}
Robin~A Moser and G{\'a}bor Tardos.
\newblock A constructive proof of the general {Lov\'asz} local lemma.
\newblock {\em J. {ACM}}, 57(2):1--15, 2010.

\bibitem{Norine06}
Serguei Norine, Paul~D. Seymour, Robin Thomas, and Paul Wollan.
\newblock Proper minor-closed families are small.
\newblock {\em J. Comb. Theory, Ser. {B}}, 96(5):754--757, 2006.

\bibitem{ollivier2006small}
Yann Ollivier.
\newblock On a small cancellation theorem of {Gromov}.
\newblock {\em Bulletin of the Belgian Mathematical Society-Simon Stevin},
  13(1):75--89, 2006.

\bibitem{osajda2020small}
Damian Osajda.
\newblock Small cancellation labellings of some infinite graphs and
  applications.
\newblock {\em Acta Mathematica}, 225(1):159 -- 191, 2020.

\end{thebibliography}
\end{document}